\documentclass{amsart}

\usepackage{wasysym}
\usepackage[usenames]{xcolor}
\usepackage{parsetree}
\usepackage{mathrsfs}
\usepackage{amssymb}
\usepackage{latexsym}
\usepackage{amsmath}
\usepackage{enumerate}
\usepackage{stmaryrd}
\usepackage{amsthm}
\usepackage{mathtools}

\newcommand{\shapipe}{\,
                    \setlength{\unitlength}{1ex}
                    \begin{picture}(2,2)
                    \put(.25,.15){\line(0,1){1.22}}
                    \put(.25,1.37){\line(1,0){0.61}}
                    \put(.45,-.05){\line(1,0){0.61}}
                    \put(.4,0){\line(1,0){0.61}}
                    \put(.35,.05){\line(1,0){0.61}}
                    \put(.3,.1){\line(1,0){0.61}}
                    \put(.25,.15){\line(1,0){0.61}}
                    \put(1.04,-.05){\line(0,1){1.22}}
                    \put(0.99,0){\line(0,1){1.22}}
                    \put(0.94,.05){\line(0,1){1.22}}
                    \put(0.89,.1){\line(0,1){1.22}}
                    \put(0.84,.15){\line(0,1){1.22}}
                    \end{picture}
                    \!}

\newcommand{\qee} {\hspace*{2mm}\hfill $\shapipe$}

\newtheorem{theorem}{Theorem}[section]
\newtheorem{define}[theorem]{Definition}

\newtheorem{exa}[theorem]{Example}
\newenvironment{example}{\begin{exa} \rm}{\qee\end{exa}}
\newtheorem{propo}[theorem]{Proposition}

\newtheorem{exerc}[theorem]{Exercise}

\newtheorem{conj}[theorem]{Conjecture}

\newtheorem{ques}[theorem]{Open Question}
\newenvironment{question}{\begin{ques} \rm}{\qee\end{ques}}
\newtheorem{lem}[theorem]{Lemma}
\newenvironment{lemma}{\begin{lem} \it}{\end{lem}}
\newtheorem{cor}[theorem]{Corollary}
\newenvironment{corollary}{\begin{cor} \it}{\end{cor}}
\newtheorem{factief}[theorem]{Fact}
\newenvironment{fact}{\begin{factief} \it}{\end{factief}}
\newtheorem{rem}[theorem]{Remark}
\newenvironment{remark}{\begin{rem} \rm}{\qee\end{rem}}
\newtheorem{dis}[theorem]{Discussion}

 \newcommand{\bident}[0]{\bigskip\noindent}
\newcommand{\medent}{\medskip\noindent}
  \newcommand{\tupel}[1]{{\langle #1 \rangle}}
\newcommand{\verz}[1]{\{ #1 \}}
\newcommand{\stackarrow}[1]{\stackrel{#1}{\longrightarrow}}
\newcommand{\To}{\Rightarrow}
\renewcommand{\iff}{\leftrightarrow}
\newcommand{\Iff}{\Leftrightarrow}
\newcommand{\braee}[1]{\, [ \! [ #1 ] \! ] \,}
 \newcommand{\nrhd}{\mathrel{\not\! \rhd}}

\newcommand{\lhdneq}{\mathrel{\lhd_{\hspace*{-0.235cm}{}_{\neq}}}}

\newcommand{\opr}{{\Square}}

\newcommand{\omodel}{$\omega$-model}
\newcommand{\pmodel}{{\sf e}-$\omega$-model}
\newcommand{\qmodel}{{\sf i}-$\omega$-model}
\newcommand{\kraj}{Kraj\'{\i}{\v{c}}ek}
\newcommand{\thres}[2]{#1{\restriction} #2}
\newcommand{\unmo}{\forall}
\newcommand{\exmo}{\exists}
 \newcommand{\gnum}[1]{\underline{ \ulcorner #1 \urcorner}}
 \newcounter{coco}
 \newcommand{\constant}[1]{\refstepcounter{coco}\label{#1}{\mathfrak c}_{\arabic{coco}}}
\newcommand{\constantref}[1]{{\mathfrak c}_{\ref{#1}}}

\title[Small is very small]{The Small-Is-Very-Small Principle}

\author{Albert Visser}
 \address{Philosophy, Faculty of Humanities,
                Utrecht University,
               Janskerkhof 13,
                3512BL~~Utrecht, The Netherlands}
\email{a.visser@uu.nl}
\date{\today}



\thanks{We thank Lev Beklemishev for enlightening discussions. We thank Ali Enayat for suggesting some important references.  We are grateful to Joost Joosten for sharing his insights on
the Friedman-Goldfarb-Harington Theorem. 
The main result of Section~\ref{degrees} is due to the previous 2006 version of me. 
I thank my previous self for its gracious permission to publish the result here.}

\begin{document}

\begin{abstract} 
The central result of this paper is \emph{the small-is-very-small principle} for restricted sequential theories. 
The principle says roughly that whenever the given theory shows that a property has a small witness,
i.e. a witness in every definable cut, then it shows that the property has a very small witness: i.e. a witness
 below a given standard number. 

We draw various consequences from the central result. For example (in rough formulations):  (i) Every restricted, recursively
enumerable sequential theory has a finitely axiomatized extension that is conservative w.r.t. formulas of complexity $\leq n$.
(ii) Every sequential model has, for any $n$,  an extension that is elementary for formulas of complexity $\leq n$,
  in which the intersection of all definable cuts is the
natural numbers.
(iii) We have reflection for $\Sigma^0_2$-sentences with sufficiently small witness in any consistent restricted theory $U$. (iv) Suppose
$U$ is recursively enumerable and sequential. Suppose further that every recursively enumerable and sequential
$V$ that locally inteprets $U$, globally interprets $U$. Then, $U$ is mutually globally interpretable
with a finitely axiomatized sequential theory.

The paper contains some careful groundwork developing partial satisfaction predicates in sequential theories for the complexity
measure \emph{depth of quantifier alternations.}
\end{abstract}

\maketitle

\section{Introduction}
Some proofs are like hollyhocks. If you are nice to them they give different flowers every year.
This paper is about one such proof. 
 I discovered it when searching for alternative, more syntactic, proofs of  certain theorems by Harvey Friedman (discussed in \cite{smor:nons85})
 and by  Jan {\kraj} (see \cite{kraj:note87}).
 The relevant theorem due to Harvey Friedman tells us that, if a finitely axiomatized, sequential, consistent theory $A$ interprets a recursively enumerable
 theory $U$, then $A$ interprets $U$ faithfully. Kraj\'{\i}\v{c}ek's theorem tells us that a finitely axiomatized, sequential, consistent theory
 cannot prove its own inconsistency on arbitrarily small cuts. There is a close connection between these two theorems.
  
The quest for a syntactic proof  succeeded and the results were reported in \cite{viss:unpr93}. 
One advantage of having such a syntactic proof is clearly that it can be `internalized' in the theories we study.
I returned to the argument in
 a later paper \cite{viss:faith05}, which contains  improvements  and, above all, a better theoretical framework. 
 In my papers  \cite{viss:inte14} and \cite{viss:arit15},  the argument is employed to prove results about 
 provability logic  and about degrees of interpretability, respectively. 
 
 The syntactic argument in question
 is a Rosser-style argument or, more specifically, a Friedman-Goldfarb-Harrington-style
argument. It has all the mystery of a Rosser argument: even if every step is completely clear, it still
retains a feeling of magic trickery. 

\subsection{Contents of the Paper}
 In the present paper, we will obtain more information from the  Friedman-Goldfarb-Harrington-style
argument discussed above.
In previous work, the basic conclusion of the argument is that, given a consistent, finitely axiomatized, sequential theory $A$, there
is an interpretation $M$ of the basic arithmetic ${\sf S}^1_2$ in $A$ that is $\Sigma^0_1$-sound. In the present paper, we extend our
scope from \emph{finitely axiomatized sequential theories} to \emph{restricted sequential theories} ---this means that we 
consider theories with axioms of complexity below
a fixed finite bound. Secondly, we replace the $\Sigma^0_1$-soundness by the more general  small-is-very-small principle (SIVS).

The improved results have a number of consequences. 
In Section~\ref{core}, we show that, for any $n$, every consistent, restricted, recursively
enumerable, sequential theory has a finitely axiomatized extension that is conservative w.r.t. formulas of complexity $\leq n$.
In Section~\ref{stare}, we show that, for any $n$, every sequential model has an elementary extension w.r.t. formulas of
  of complexity $\leq n$, such that the intersection of all definable cuts consists of the standard numbers.
  In Section~\ref{refl}, we indicate how results concerning $\Sigma^0_2$-soundness can be derived from our
  main theorem. 
Finally, in Section~\ref{degrees}, we prove a result in the structure of the combined degrees of
local and global interpretability of recursively enumerable, sequential theories. We show that
\emph{if} a local degree contains a minimal global degree, \emph{then} this global degree contains a finitely axiomatized
theory. Thus, finite axiomatizability has a natural characterization, modulo global interpretability, in terms of the double degree structure.\footnote{I presented
this result in a lecture for the Moscow Symposium 
on Logic, Algebra and Computation in 2006. However, I was not able to write down the proof, since I lacked the 
necessary groundwork on partial satisfaction. This groundwork is provided in Section~\ref{bafa} of the present paper.}

Section~\ref{bafa} provides the necessary elementary facts. Unlike similar sections in other papers of mine, this section also
contains something new. In \cite{viss:unpr93}, I provided groundwork for the development of partial satisfaction predicates for the
complexity measure \emph{depth of quantifier alternations}. Our present Subsection~\ref{compmeas} gives a much better treatment
of the complexity measure than the one in  \cite{viss:unpr93}. Subsection~\ref{sare} develops the facts about sequential theories
and partial satisfaction predicates in greater detail than previously available in the literature. Moreover, we provide careful
estimates of the complexities yielded by the various constructions. 
On the one hand, these subsections
contain `what we already knew', on the other hand, as I found, even if you already know how things go, it can still be quite a puzzle
to get all nuts and bolts at the precise places where they have to go. Of course, the present treatment is still not fully explicit, but we are
further on the road.

Section~\ref{rost} contains the central result of the paper. As the reader will see, after all is said and done, the central argument is amazingly simple.
The work is in creating the setting in which the result can be comfortably stated.

\section{Basic Notions and Facts}\label{bafa}
In the present section, we provide the basics needed for the rest of the paper. 
As pointed out in the introduction the development of partial
satisfaction predicates is done in more detail here than elsewhere. For this reason this
section may also turn out to be useful for subsequent work. Of course, the reader who wants
to get on quickly to more exciting stuff could briefly look over the relevant subsections and, if needed,
return to them later.

\subsection{Theories}
In this paper we will study theories with finite signature.
In most of our papers, theories are intensional objects equipped
with a formula representing the axiom set. In the present paper, to the contrary,
a theory is just a set of sentences of the given signature closed under deduction.
This is because most of the results in the paper are extensional.

Also we do not have any constraints on the complexity of the axiom set of the
theory. If a theory is finitely axiomatizable, \emph{par abus de langage}, we use the variables like $A$ and $B$ for it, making the
letters do double work: they both stand for the theory and for a single axiom. 

When we diverge from our general format this will always be explicitly mentioned. 

\medent
In the paper, we will meet many concrete theories, to wit {\sf AS}, ${\sf PA}^-$, {\sf EA}, {\sf PRA}, {\sf PA}.
We refer the reader to the textbooks~\cite{haje:meta91} and \cite{kaye:mode91} for an introduction to these theories.

\subsection{Translations and Interpretations}\label{interpretations}

We present the notion of \emph{$m$-dimensi\-onal interpretation without parameters}.
There are two extensions of this notion: we can  consider piecewise interpretations
and we can add parameters. We will give some details on parameters in Appendix~\ref{machosmurf}.
We will not describe piecewise interpretations here.

Consider two signatures $\Sigma$ and $\Theta$. An $m$-dimensional translation $\tau:\Sigma \to \Theta$
is a quadruple $\tupel{\Sigma,\delta,\mathcal F,\Theta}$, where $\delta(v_0,\ldots,v_{m-1})$ is a $\Theta$-formula and where,
for any $n$-ary predicate $P$ of $\Sigma$, $\mathcal F(P)$ is a formula $A(\vec v_{0},\ldots, \vec v_{n-1})$ in the language
of signature $\Theta$, where
$\vec v_i = v_{i0},\ldots, v_{i(m-1)}$. Both in the case of $\delta$ and $A$ all free variables are among the variables shown.
Moreover, if $i\neq j$ or $k \neq\ell$, then $v_{ik}$ is syntactically different from $v_{j\ell}$.

We demand that we have $\vdash {\mathcal F}(P)(\vec v_0,\ldots, \vec v_{n-1}) \to \bigwedge_{i<n} \delta(\vec v_i)$.
Here $\vdash$ is provability in predicate logic. This demand is inessential, but it is convenient to have.

We define $B^\tau$ as follows:
\begin{itemize}
\item
$(P(x_0,\ldots,x_{n-1}))^\tau :=  \mathcal F(P)(\vec x_0,\ldots, \vec x_{n-1})$.
\item
$(\cdot)^\tau$ commutes with the propositional connectives.
\item
$(\forall x\, A)^\tau := \forall \vec x \,( \delta(\vec x\,) \to A^\tau)$.
\item
$(\exists x\, A)^\tau := \exists \vec x \,( \delta(\vec x\,) \wedge A^\tau)$.
\end{itemize}
There are two worries about this definition. 
First, what variables $\vec x_i$ on the side of the translation $A^\tau$
correspond with $x_i$ in the original formula $A$? 
 The second worry is that  substitution of variables in $\delta$
and $\mathcal F(P)$ may cause variable-clashes.
These worries are never important in practice: we choose
 `suitable' sequences $\vec x$ to correspond to variables $x$, and we avoid clashes by $\alpha$-conversion. 
 However, if we want to give precise definitions of translations and, for example, of composition of translations, these
 problems come into play. The problems are clearly solvable in a systematic way, but this endeavor is beyond the scope of this paper.
 
 We allow the  identity predicate to be translated to a formula that is not identity.
 
 \medent
 A translation $\tau$ is \emph{direct}, if it is one-dimensional and if $\delta_\tau(x) := (x=x)$ and if it translates
identity to identity.

\medent
There are several important operations on translations. 
\begin{itemize}
\item
${\sf id}_\Sigma$ is the identity translation. We take $\delta_{{\sf id}_\Sigma}(v) := v=v$ and $\mathcal F(P) := P(\vec v\,)$.
\item
We can compose translations. Suppose $\tau:\Sigma\to \Theta$ and $\nu:\Theta \to \Lambda$.
Then $\nu\circ \tau$ or $\tau\nu$ is a translation from $\Sigma$ to $\Lambda$. We define:
\begin{itemize}
\item
$\delta_{\tau\nu}(\vec v_0,\ldots,\vec v_{m_\tau-1}) := \bigwedge_{i< m_\tau}\delta_\nu(\vec v_i) \wedge (\delta_\tau(v_0,\ldots,v_{m_\tau-1}))^\nu$.
\item
$P_{\tau\nu}(\vec v_{0,0},\ldots,\vec v_{0,m_\tau-1},\ldots \vec v_{n-1,0},\ldots,\vec v_{n-1,m_\tau-1}) := \\
{\bigwedge_{i<n, j< m_\tau} \delta_\nu(\vec v_{i,j})}\, \wedge (P(v_0,\ldots, v_{n-1})^\tau)^\nu $.
\end{itemize}
\item
Let  $\tau,\nu :\Sigma \to \Theta$ and let $A$ be a sentence of signature $\Theta$. We define the disjunctive translation 
$\sigma:= \tau\tupel{A}\nu:\Sigma\to \Theta$ as follows. We take $m_{\sigma}:= {\sf max}(m_\tau,m_\nu)$.
We write $\vec v\restriction n$, for the restriction of $\vec v$ to the first $n$ variables, where
$n\leq {\sf length}(\vec v)$.
\begin{itemize}
\item
$\delta_\sigma(\vec v) := 
(A\wedge \delta_\tau(\vec v\restriction m_\tau)) \vee 
(\neg \, A\wedge\delta_\nu( \vec v \restriction m_\nu))$. 
\item
$P_\sigma(\vec v_0,\ldots,\vec v_{n-1}) := (A\wedge P_\tau(\vec v_0\restriction m_\tau,\ldots,\vec v_{n-1}\restriction m_\tau))
\vee \\ \hspace*{3.1cm} (\neg \, A\wedge P_\nu(\vec v_0\restriction m_\nu,\ldots,\vec v_{n-1}\restriction m_\nu))$
\end{itemize}
\end{itemize}
Note that in the definition of $\tau\tupel{A}\nu$ we used a padding mechanism. In case, for example, $m_\tau < m_\nu$,
the variables $v_{m_\tau},\ldots, v_{m_\nu-1}$ are used `vacuously' when we have $A$.
If we had piecewise interpretations, where domains are built up from pieces with possibly different dimensions,
 we could avoid padding by building the domain directly of disjoint pieces
with different dimensions.

\medent
A translation relates signatures; an interpretation relates theories.
An interpretation $K:U\to V$ is a triple $\tupel{U,\tau,V}$, where $U$ and $V$ are theories and $\tau:\Sigma_U \to \Sigma_V$.
We demand: for all theorems $A$ of $U$, we have $V\vdash A^\tau$.
Here are some further definitions.

\begin{itemize}
\item
${\sf ID}_U:U\to U$ is the interpretation $\tupel{U,{\sf id}_{\Sigma_U},U}$.
\item
Suppose $K:U\to V$ and $M:V\to W$. Then, $KM := M\circ K:U\to W$ is $\tupel{U,\tau_M\circ \tau_K,W}$.
\item
Suppose $K:U \to (V+A)$ and $M:U\to (V+\neg\,A)$. Then $K\tupel{A}M:U\to V$ is the interpretation
$\tupel{U,\tau_K\tupel{A}\tau_M,V}$. In an appropriate category $K\tupel{A}M$ is a special case of a product.
\end{itemize}

\noindent
A translation $\tau$ maps a model $\mathcal M$ to an internal model $\widetilde \tau (\mathcal M)$ provided that 
$\mathcal M \models \exists \vec x\, \delta_\tau(\vec x\,)$. Thus, an interpretation $K:U\to V$ gives us a
mapping $\widetilde K$ from ${\sf MOD}(V)$,  the class of models of $V$, to ${\sf MOD}(U)$, the class
of models of $U$. If we build a category of theories and interpretations, usually {\sf MOD}
with ${\sf MOD} (K):= \widetilde K$ will be a contravariant functor.

We use $U \stackarrow{K} V$ or  $K:U \lhd V$ or $K:{V\rhd U}$ as alternative notations for $K:U \to V$.
The alternative notations $\lhd$ and $\rhd$ are used in a context where we are interested in interpretability
as a preorder or as a provability analogue. 

We write: $U \lhd V$ and $U \rhd V$, for: there is an interpretation $K:U \lhd V$. We use $U \equiv V$, for:
$U \lhd V$ and $U\rhd V$.

The arrow notations are mostly used in a context where we are interested in a category of interpretations, but also
simply when they improve readability.

\medent
We write $U \lhd_{\sf loc} V$ or $V \rhd_{\sf loc} U$ for: for all finite subtheories $U_0$ of $U$,
 $U_0 \lhd V$. We pronounce this as: $U$ is locally interpretable in $V$ or
 $V$ locally interprets $U$. We use $\equiv_{\sf loc}$ for the induced equivalence relation  of $\lhd_{\sf loc}$.

  \subsection{Complexity and Restricted Provability}\label{compmeas}

{\em Restricted provability} plays an important role in the study of interpretability between sequential theories. 
An $n$-proof is a proof from axioms with G\"odel number smaller or equal than $n$\/ only involving formulas
of complexity smaller or equal than $n$. To work conveniently with this notion, 
a good complexity measure is needed. Such a measure should satisfy
three conditions. 
\begin{enumerate}[i.]
\item
Eliminating terms in favor of a relational formulation should 
raise the complexity only by a fixed standard number. 
\item
Translation of a formula via the translation $\tau$ should raise the complexity of the formula by a fixed
standard number depending only on $\tau$. 
\item
 The tower of exponents involved in cut-elimination should be of height linear in the complexity of
the formulas involved in the proof. 
\end{enumerate}

\noindent
Such a good measure of complexity together with a verification of desideratum (iii) ---a form of nesting degree of quantifier alternations--- is supplied in the
work of Philipp Gerhardy. See \cite{gerh:refi03} and \cite{gerh:role05}. A slightly different measure is  provided by 
Samuel Buss in  \cite{buss:situ15}. Buss also proves that (iii) is fulfilled for his measure. In fact, Buss proves a sharper result. He shows that
 the bound is $d+O(1)$ for $d$ alternations.
In the present paper,  we will follow Buss' treatment.

We work over a signature $\Theta$. The formula-classes we define are officially called $\Sigma^\ast_n(\Theta)$ and 
$\Pi^\ast_n(\Theta)$. However, we will suppress the $\Theta$ when it is clear from the context. 
Let {\sf AT} be the class of atomic formulas for $\Theta$, extended with $\top$ and $\bot$. We define:
{\small
\begin{itemize}
\item
$\Sigma^\ast_{0} := \Pi^\ast_{0} := \emptyset$.
\item
$\Sigma^\ast_{n+1} :: =\\  {\sf AT} \mid    \neg \,\Pi^\ast_{n+1} \mid (\Sigma^\ast_{n+1} \wedge \Sigma^\ast_{n+1}) \mid
(\Sigma^\ast_{n+1} \vee \Sigma^\ast_{n+1}) \mid (\Pi^\ast_{n+1} \to \Sigma^\ast_{n+1}) 
\mid \exists v \, \Sigma^\ast_{n+1} \mid \forall v\, \Pi^\ast_{n}$.
\item
$\Pi^\ast_{n+1} :: = \\ {\sf AT} \mid   \neg \,\Sigma^\ast_{n+1} \mid (\Pi^\ast_{n+1} \wedge \Pi^\ast_{n+1}) \mid
(\Pi^\ast_{n+1} \vee \Pi^\ast_{n+1}) \mid (\Sigma^\ast_{n+1} \to \Pi^\ast_{n+1}) \mid \forall v\, \Pi^\ast_{n+1}
\mid \exists v\, \Sigma^\ast_{n}$.
\end{itemize}
}

\medent
Buss uses $\Sigma_{n+1}$ and $\Pi_{n+1}$ where we use $\Sigma^\ast_{n+1}$ and $\Pi^\ast_{n+1}$. We employ the asterix to avoid
confusion with the usual complexity classes in the arithmetical hierarchy where bounded quantifiers
also play a role. Secondly, we modified Buss' inductive definition a bit
in order to get unique generation histories. For example, Buss adds $\Pi^\ast_{n}$ to $\Sigma^\ast_{n+1}$ in stead of $\forall v\, \Pi^\ast_{n}$.
In addition our $\Sigma^\ast_0$ and $\Pi^\ast_0$ are empty, where Buss' corresponding classes consist of the quantifier-free formulas.

Here is the parse-tree of $\forall x\, (\forall y\, \exists z\, Pxyz \to \exists u\,\exists v\, Qxuv)$ as an element of $\Sigma^\ast_4$.

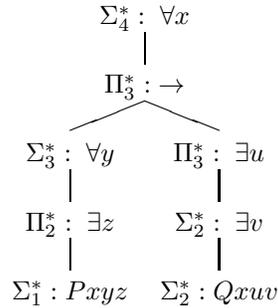
\[
\ptbegtree
\ptbeg
\ptnode{$\Sigma^\ast_4:\;  \forall x$}
\ptbeg
\ptnode{$\Pi^\ast_3: \;   \to$}
\ptbeg \ptnode{$\Sigma^\ast_3: \;  \forall y$}
\ptbeg \ptnode{$\Pi^\ast_2:\; \exists z$}
\ptleaf{$\Sigma^\ast_1: Pxyz$}
\ptend
\ptend  
\ptbeg
\ptnode{$\Pi^\ast_3: \;   \exists u$}
\ptbeg
\ptnode{$\Sigma^\ast_2: \; \exists v$}
\ptleaf{$\Sigma^\ast_2: Qxuv$}
\ptend
\ptend
\ptend
\ptend
\ptendtree
\]

\medent
The extensional equivalence, for $n>0$ of our definition to Buss's is immediate from the following:

\begin{fact}\label{smallfact}
The quantifier-free formulas are in $\Sigma^\ast_1\cap\Pi^\ast_1$ and \\
 $ \Sigma^\ast_n \cup \Pi^\ast_n \subseteq \Sigma_{n+1}^\ast \cap \Pi_{n+1}^\ast$.
\end{fact}

\noindent
The proof is by five simple inductions.
We define:
{\small
\begin{itemize}
\item
$\Delta^\ast_{n+1} :: =\\  {\sf AT} \mid
 \neg \,\Delta^\ast_{n+1} \mid (\Delta^\ast_{n+1} \wedge \Delta^\ast_{n+1}) \mid
(\Delta^\ast_{n+1} \vee \Delta^\ast_{n+1}) \mid (\Delta^\ast_{n+1} \to \Delta^\ast_{n+1}) \mid  \exists v\, \Sigma^\ast_n \mid \forall v\, \Pi^\ast_n$.
\end{itemize}
}

\noindent 
We have:

\begin{theorem}
$\Delta_{n+1}^\ast = \Sigma_{n+1}^\ast \cap \Pi_{n+1}^\ast$.
\end{theorem}

\begin{proof}
That $\Delta_{n+1}^\ast \subseteq  \Sigma_{n+1}^\ast \cap \Pi_{n+1}^\ast$ is an easy induction based on Fact~\ref{smallfact}.
We prove the converse by ordinary induction on formulas. The atomic case and the propositional cases are immediate.
Suppose $A$  in $\Sigma_{n+1}^\ast \cap \Pi_{n+1}^\ast$  has the form $\exists v\, B$. Then
$B$ must be in $\Sigma^\ast_n$. It follows that $A$ is in $\Sigma^\ast_n$ and, thus, that $A$ is in $\Delta^\ast_{n+1}$.
\end{proof}

\noindent
We want a complexity measure $\rho(A)$ such that $\rho(A)$ is the smallest $n$ such that $A$ 
is in $\Sigma_n^\ast$. This measure is very close to the measure
that was employed in \cite{viss:unpr93}. We recursively define this measure by taking $\rho := \rho_{\exmo}$, where  
$\rho_{\exmo}$ is defined as follows:

\begin{itemize}
\item
$\rho_{\exmo}(A) := \rho_{\unmo}(A)= 1$, if $A$ is atomic.  
\item
$\rho_{\exmo}(\neg\, B) := \rho_{\unmo}(B)$, $\rho_{\unmo}(\neg\, B) := \rho_{\exmo}(B)$.
\item
$\rho_{\exmo}(B\wedge C) := {\sf max}(\rho_{\exmo}(B),\rho_{\exmo}(C))$, 
$\rho_{\unmo}(B\wedge C) := {\sf max}(\rho_{\unmo}(B),\rho_{\unmo}(C))$.
\item
$\rho_{\exmo}(B\vee C) := {\sf max}(\rho_{\exmo}(B),\rho_{\exmo}(C))$, 
$\rho_{\unmo}(B\vee C) := {\sf max}(\rho_{\unmo}(B),\rho_{\unmo}(C))$.
\item
$\rho_{\exmo}(B\to C) := {\sf max}(\rho_{\unmo}(B),\rho_{\exmo}(C))$, 
$\rho_{\unmo}(B\to C) := {\sf max}(\rho_{\exmo}(B),\rho_{\unmo}(C))$.
\item
$\rho_{\exmo}(\exists v\, B) := \rho_{\exmo} (B)$, 
$\rho_{\unmo}(\exists v\, B) := \rho_{\exmo}(B)+1$.
\item
$\rho_{\exmo}(\forall v\, B) :=  \rho_{\unmo}(B)+1$,
$\rho_{\unmo}(\forall v\, B) := \rho_{\unmo}(B)$.
\item
$\rho(A) := \rho_{\exmo}(A)$, $\rho_0(A) := {\sf max}(\rho_{\exmo}(A),\rho_{\unmo}(A))$.
\end{itemize}

\noindent
We verify the basic facts about $\rho$.

\begin{theorem}\label{fruitigesmurf}
$\rho_{\unmo} (A) \leq \rho_{\exmo}(A)+1$ and $\rho_{\exmo} (A) \leq \rho_{\unmo}(A)+1$.
\end{theorem}

\begin{proof}
The proof is by induction on $A$. We treat the case that $A = \exists v\, B$.
We have:
 $\rho_{\unmo} (\exists v\, B) =  \rho_{\exmo} (B)+1 = \rho_{\exmo}(\exists v\, B)+1$.
 Note that this does not use the induction hypothesis.
\end{proof}

\begin{theorem}
$\Sigma^\ast_n = \verz{A \mid \rho_{\exmo}(A) \leq n}$ and $\Pi^\ast_n = \verz{A \mid \rho_{\unmo}(A) \leq n}$.
It follows that, for $n>0$, we have $\Delta^\ast_n = \verz{A \mid \rho_0(A) \leq n}$
\end{theorem}

\begin{proof}
We  prove, by induction on $n$, that:
$A\in \Sigma^\ast_n$ iff $\rho_{\exmo}(A) \leq n$ and $A\in \Pi^\ast_n$ iff $\rho_{\unmo}(A) \leq n$. 

\medent
The case of 0 is clear.
We prove by induction on the definition of $\Sigma^\ast_{n+1}$, that $A \in \Sigma^\ast_{n+1}$ iff $\rho_\exmo(A)\leq n+1$.
The atomic case, the propositional cases and the existential case are clear. Suppose $A = \forall v\, B$.
If $A$ is in $A \in \Sigma^\ast_{n+1}$, then $B$ is in $\Pi^\ast_n$. By the Induction Hypothesis, $\rho_{\unmo}(B) \leq n$,
so $\rho_{\exmo}(A) \leq n+1$. If $\rho_{\exmo}(A) \leq n+1$, then $\rho_{\unmo}(B) \leq n$. Hence, by the Induction Hypothesis,
$B \in \Pi^\ast_{n}$, so $A \in \Sigma^\ast_{n+1}$. The case of $\Pi^\ast_{n+1}$ is similar.
\end{proof}

\noindent
Let $\tau: \Sigma \to \Theta$ be a translation. We define $\rho^\ast(\tau)$ to be the maximum
 of $\rho_0(\delta_\tau)$ and the
  the $\rho_0 (P_\tau)$, for $P$ in $\Sigma$.
  If $K$ is an interpretation, then $\rho^\ast(K)  := \rho^\ast(\tau_K)$. 
  
\begin{theorem}
Let $\tau: \Sigma \to \Theta$.
We have: \\
$\rho_{\exmo}(A^\tau)  \leq \rho_{\exmo}(A) + \rho^\ast(\tau)$ and
$\rho_{\unmo}(A^\tau) \leq \rho_{\unmo}(A) + \rho^\ast (\tau)$.
\end{theorem}

\begin{proof}
The proof is by induction on $A$.
The case of the atoms is trivial. 

\medent
We treat the case of implication and $\rho_{\exmo}$.
Suppose $A$ is $B\to C$. We have:
 \begin{eqnarray*}
  \rho_{\exmo}(A^\tau)   & = &  {\sf max}(\rho_{\unmo}(B^\tau),\rho_{\exmo}(C^\tau)) \\
  & \leq & {\sf max}(\rho_{\unmo}(B) + \rho^\star(\tau), \rho(C) + \rho^\ast(\tau)) \\
  & = &  {\sf max}(\rho_{\unmo}(B),\rho_{\exmo}(C)) + \rho^\ast(\tau) \\
  & = & \rho_{\exmo}(A) + \rho^\ast(\tau)
\end{eqnarray*}
The other cases concerning the propositional connectives are similar.

\medent
We treat the case for universal quantification and $\rho_{\exmo}$.
Suppose $A$ is $\forall v\, B$. We have:
 \begin{eqnarray*}
  \rho_{\exmo}(A^\tau)   & = &  \rho_{\exmo} (\forall \vec v\, (\delta_\tau(\vec v) \to B^\tau)) \\ 
& = &  \rho_{\unmo} (\delta_\tau(\vec v) \to B^\tau) +1 \\
& = & {\sf max}(\rho_{\exmo}(\delta_\tau(\vec v)), \rho_{\unmo} (B^\tau))+1 \\
& \leq & \rho_{\unmo}(B) + \rho^\ast(\tau)+1 \\
& = & \rho_{\exmo}(A) + \rho^\ast (\tau) 
\end{eqnarray*}
The remaining cases for the quantifiers are similar or easier.
\end{proof}

\subsection{Sequential Theories}

The notion of sequentiality is due to Pavel Pudl\'ak. 
See, e.g., \cite{pudl:prime83}, \cite{pudl:cuts85},  \cite{myci:latt90}, \cite{haje:meta91}.

To define sequentiality we use the auxiliary theory
 ${\sf AS}^+$ (Adjunctive Set Theory with extras).
The signature $\mathfrak A$ of ${\sf AS}^+$ consists of unary predicate symbols {\sf N} and {\sf Z},
binary predicate symbols $\in$, {\sf E}, $\leq$, $<$, {\sf S}, ternary predicate symbols
{\sf A} and {\sf M}. 
\begin{enumerate}[${\sf AS}^+$1]
\item
We have a set of axioms that provide a relative interpretation $\mathscr N$ of ${\sf S}^1_2$ in ${\sf AS}^+$, 
where {\sf N} represents the natural numbers,
{\sf E} represents numerical identity, {\sf Z} stands for zero modulo {\sf E}, {\sf A} stands for
addition modulo {\sf E}, and {\sf M} stands for multiplication modulo {\sf E}.
\item
 $\vdash \exists x \, \forall y\; y\not\in x$,
 \item
  $\vdash \forall x \, \exists y \, \forall z\; (z \in y \iff z = x)$,
  \item
  $\vdash \forall x, y\,  \exists z \, \forall u\, (u\in z \iff (u \in x \vee u \in y))$,
   \item
  $\vdash \forall x, y\,  \exists z \, \forall u\, (u\in z \iff (u \in x \wedge u \in y))$,
  \item
  $\vdash \forall x, y\,  \exists z \, \forall u\, (u\in z \iff (u \in x \wedge \neg \, u \in y))$.
  \end{enumerate}

\noindent
An important point is that we do not demand extensionality for our sets. A many-sorted version
of ${\sf AS}^+$ would be somewhat more natural. We refrain from developing it in this way here
to avoid the additional burden of working with interpretations between many-sorted theories.

A theory is sequential iff it  interprets the theory ${\sf AS}^+$
via a direct interpretation $\mathcal S$. We call such an $\mathcal S$ a \emph{sequence scheme}.

It is possible to work with an even simpler base theory.
The theory  ${\sf AS}$ is given by the following axioms. 
\begin{enumerate}[{\sf AS}1]
\item $\vdash \exists y \, \forall x\; x\not\in y$,
\item $\vdash \forall x\;\forall y \, \exists z\,\forall u\;(u\in z \iff (u\in y\vee u=x))$.
\end{enumerate}

\noindent
One can show that {\sf AS} is mutually directly interpretable with ${\sf AS}^+$.
 For details concerning the bootstrap see e.g. the textbook \cite{haje:meta91}
and also \cite{myci:latt90},  \cite{viss:card09}, \cite{viss:hume11}, \cite{viss:what13}.

\begin{remark}
We could work in a somewhat richer class of theories, the \emph{polysequential} theories.
See \cite{viss:what13}.

Let's say that an interpretation is \emph{$m$-direct}, if it is $m$-dimensional, if
its domain consists of all $m$-tuples of the original domain, and if identity is interpreted
as component-wise identity. A theory $U$ is $m$-sequential, if there is an $m$-direct
interpretation of {\sf AS} in $U$. A theory is polysequential, if it is $m$-sequential for 
some $m\geq 1$. Note that if we want the ${\sf AS}^+$ format, the interpretation of the
natural numbers should also be chosen to be $m$-dimensional for the given $m$.
The development given in the present paper also works with minor adaptations in the 
 polysequential case. 

It is known that there are polysequential theories that are not sequential. However, I only have
an artificial example. Every polysequential theory is polysequential without parameters, where
a sequential theory may essentially need an interpretation with parameters to witness its
sequentiality. (One raises the dimension to `eat up' the parameters.) Polysequential theories
are closed under bi-interpretability. Moreover, every polysequential theory is bi-interpretable with
a sequential one.
\end{remark}

\subsection{Satisfaction {\&} Reflection}\label{sare}
In this subsection, we develop partial satisfaction predicates for sequential theories with some care.
We prove the corresponding partial reflection principles. This subsection is rather long
 because it provides many details. The impatient reader could choose to proceed to Theorem~\ref{summumsmurf}, since
 that is the main result of the subsection that we will use in the rest of the paper. 

\medent
Consider any signature $\Theta$. We extend the signature $\mathfrak A$ of ${\sf AS}^+$ in a disjoint way
with $\Theta$ to, say, $\mathfrak A + \Theta$. Call the resulting theory (without any new axioms) ${\sf AS}^+(\Theta)$.

We work towards the definition of partial satisfaction predicates, We provide a series of definitions illustrative of what 
we need to get off the ground.
\begin{itemize}
\item
${\sf pair}(u,v,w) :\iff \exists a\, \exists b\, (\forall c \, (c \in w \iff (c=a \vee c = b)) \wedge \\
\hspace*{3.45cm}
 \forall d\, (d\in a \iff d=u) \wedge \forall e\, (e \in b \iff (e= u \vee e= v)))$.\\
 We can easily show that for all $u$ and $v$ there is a $w$ such that ${\sf pair}(u,v,w)$ and,
 whenever ${\sf pair}(u,v,w)$ and ${\sf pair}(u',v',w)$, then $u=u'$ and $v=v'$.
 Note that there may be several $w$ such that ${\sf pair}(u,v,w)$.
 \item
$ {\sf Pair}(w) :\iff \exists u\, \exists v\, {\sf pair}(u,v,w)$.
\item
$\pi_0(w,u) :\iff \exists v\, {\sf pair}(u,v,w)$. 
\item
$\pi_1(w,v) :\iff \exists u\, {\sf pair}(u,v,w)$.
\item
${\sf fun}(f) :\iff \forall w\in f\, ({\sf pair}(w) \wedge   \forall w'\in f\, \forall u \, (\pi_0(w,u) \wedge \pi_0(w', u) \to w=w'))$.\\
Note that we adapt the notion of function to our non-extensional pairing. We demand that there is at most one
witnessing pair for a given argument.  This choice makes resetting a function on an argument where it is defined
a simple operation: we subtract one pair and we add one.
  \item
  ${\sf dom}(f,x) :\iff {\sf fun}(f) \wedge \exists w\in f\, \pi_0(w,x)$.
   \item
  $f(u) \approx v :\iff {\sf fun}(f) \wedge \exists w\in f\, {\sf pair}(u,v,w)$.
\item
${\sf nfun}(\alpha) :\iff \forall w\in \alpha\, ({\sf pair}(w) \wedge  \forall u\, ( \pi_0(w,u) \to {\sf N}(u)) \wedge \\
\hspace*{1.8cm}
  \forall w'\in \alpha\, \forall u\, \forall u' \, ((\pi_0(w,u) \wedge \pi_0(w', u') \wedge u\mathrel{\sf E} u' ) \to w=w'))$.\\
  We note that ${\sf nfun}(f)$ implies ${\sf fun}(f)$.
   \item
  ${\sf ndom}(\alpha,a) :\iff {\sf nfun}(\alpha) \wedge \exists w\in f\, \exists b\, (a\mathrel{\sf E} b \wedge \pi_0(w,b))$.
  \item
  We fix a parameter $x^\star$.\\
  $\alpha[a] \approx v :\iff {\sf nfun}(\alpha) \wedge (\exists w \in \alpha\,  \exists b\, (a\mathrel{\sf E} b \wedge {\sf pair}(b,v,w))
  \vee \\
  \hspace*{7cm} (\neg\, {\sf ndom}(\alpha,a) \wedge v= x^\star))$.\\
  So outside of {\sf ndom} we 
  set the value of $\alpha$ to a default value. In this way
  we made it a total function on the natural numbers.\footnote{The need for a parameter is regrettable but in a sequential theory
  there need not be definable elements. Of course, we could set the value at $a$ outside the {\sf ndom} of $\alpha$ to
  $\alpha$, but that would mean that when we reset our function we would change all the values outside the {\sf ndom} too.
  One way to eliminate the parameter would be to make the default value an extra part of the data for the function. Then,
  a reset would keep the default value in tact. The option of working with partial functions is certainly feasible. However, e.g., clause (e) of the
  definition of adequate set would be more complicated.}
    \item
$\alpha\braee{a}\beta :\iff  \forall b \, ({\sf ndom}(\beta, b) \iff ({\sf ndom}(\alpha,b) \vee b \mathrel{\sf E} a  )) \wedge \\
\hspace*{1.7cm}
\forall b \, \forall x\, (({\sf N}(b) \wedge \neg\, a \mathrel{\sf E} b) \to (\alpha[b] \approx x \iff \beta[b]\approx x))$.\\
\item
$\alpha [a:y] \approx \beta :\iff \alpha \braee{a} \beta \wedge \beta[a] \approx y$.
\end{itemize}
 
\noindent We will use 
$\alpha$, $\beta$ to range over functions with numerical domains (elements of {\sf nfun}).
We  note that ${\sf AS}^+(\Theta)$ proves that the `operation' $\alpha \mapsto \alpha[a:y]$ is total
and that any output sequences are extensionally the same.

We employ a usual efficient coding of syntax in the interpretation $\mathscr N$ of ${\sf S}^1_2$. We have shown above
how to formulate things in order to cope with the fact that each number, set, pair, function and sequence can have several representatives.
The definition of satisfaction would be completely unreadable if we tried to adhere to this high
standard. Hence, we will work more informally pretending, for example, that each number has just one
representative. An assignment will simply be a numerical function where we restrict our attention to the
codes of variables in the domain.

\medent
A \emph{sat-sequence} is a triple of the form
$\tupel{i, \alpha, A}$, where: 
$i$ is $+$ or $-$ (coded as, say, 1 and 0), $\alpha$ is an assignment and $A$ is a formula.
An \emph{$a$-sat-sequence} is a sat-sequence
$\tupel{i, \alpha, A}$, where  $A$ is $\Sigma^\ast_a$, if $i=+$ and  $A$ is $\Pi^\ast_a$, if $i=-$.
We call the virtual class of all $a$-sat sequences $\mathcal K_a$.

A set $X$ is \emph{good} if, for all  numbers $b$ the virtual class $X\cap \mathcal K_b$ exists as a set.
We have:

\begin{lemma}{$({\sf AS}^+(\Theta))$.}
The good sets are closed under the empty sets, singletons, union, intersection and subtraction and are downwards closed
w.r.t. the subset ordering.
\end{lemma}

\begin{proof}
We leave the easy proof to the reader.
\end{proof}
 
\noindent
We count two sat-sequences $\tupel{i,\alpha,A}$ and $\tupel{j,\beta,  B}$ as \emph{extensionally equal}
if  (i) $i$ and $j$ are
{\sf E}-equal,
(ii) $\alpha$ and $\beta$ have the same functional behaviour on the natural numbers {\sf N}, 
and (iii) $A$ and $B$ are {\sf E}-equal. We say that a sequence $\sigma$ is \emph{of the form}
$\tupel{i,\alpha,A}$ if it is extensionally equal to a sequence $\tau$ with $\tau_0 = i$, $\tau_1 = \alpha$  and $\tau_2 = A$.

We  define \emph{$n+1$-adequacy} and ${\sf sat}_n$ by external recursion on $n$.
We define ${\sf sat}_0(+,\alpha, A) :\iff \bot$ and ${\sf sat}_0(-,\alpha, A) :\iff \top$.
We define ${\sf sat}_{n+1}(i, \alpha,A)$ iff, for some $n+1$-adequate set $X$, we have
$\tupel{i, \alpha,A} \in X$.
A set $X$ is \emph{$n+1$-adequate}, if it is good, if  its elements are sat-sequences, and if it satisfies
 the following clauses: 
\begin{enumerate}[a.]
\item
If  a sequence of the form $\tupel{+,\beta,P(v_0,v_1)}$ is in $X$, then
$P(\beta(v_0),\beta(v_1))$.\footnote{Our variables are coded as numbers. Conceivably not
all numbers code variables. The values of $\beta$ on non-variables are simply \emph{don't care}.}
\\ Similarly, for other atomic formulas including
$\top$ and $\bot$.
\item
If  a sequence of the form $\tupel{-,\beta,P(v_0,v_1)}$ is in $X$, then
$\neg\, P(\beta(v_0),\beta(v_1))$.\\ Similarly, for other atomic formulas including
$\top$ and $\bot$.
\item
If  a sequence of the form $ \tupel{+,\beta,\neg\, B}$ is in $X$, then
a sequence of the form $ \tupel{-,\beta, B}$ is in $X$.
\item
If a sequence of the form $\tupel{-,\beta,\neg\, B}$ is in $X$, then
 a sequence of the form $ \tupel{+,\beta, B}$ is in $X$.
  \item
  If a sequence of the form $\tupel{+,\beta,(B\wedge C)}$ is in $X$, then
  sequences of the form $\tupel{+,\beta,B}$ and
  $ \tupel{+,\beta,C}$ are in $X$.
    \item
  If a sequence of the form $\tupel{-,\beta,(B\wedge C)}$ is in $X$, then
 a sequence of the form $\tupel{-,\beta,B}$ is in $X$ or
a sequence of the form  $  \tupel{-,\beta,C}$ is in $X$.
   \item
  If a sequence of the form $\tupel{+,\beta,(B\vee C)}$ is in $X$, then
 a sequence of the form $  \tupel{+,\beta,B}$ is in $X$ or
  a sequence of the form $  \tupel{+,\beta,C}$ is in $X$.
   \item
  If a sequence of the form $\tupel{-,\beta,(B\vee C)}$ is in $X$, then
  sequences of the form $  \tupel{-,\beta,B}$ and
  $  \tupel{-,\beta,C}$ are in $X$.
   \item
  If a sequence of the form $ \tupel{+,\beta,(B\to C)}$ is in $X$, then
  a sequence of the form $ \tupel{-,\beta,B}$ is in $X$ or
a sequence of the form  $ \tupel{+,\beta,C}$ is in $X$. 
  \item
  If a sequence of the form $\tupel{-,\beta,(B\to C)}$ is in $X$, then
   sequences of the form $  \tupel{+,\beta,B}$ and
  $  \tupel{-,\beta,C}$ are in $X$.
   \item
  If a sequence of the form $\tupel{+,\beta,\exists v\, B}$ is in $X$, then,
   for some $\gamma$ with $\beta\braee{v} \gamma$, a sequence of the form $\tupel{+,\gamma,B}$ is in $X$.
   \item
  If a sequence of the form $\tupel{-,\beta,\exists v\, B}$ is in $X$, then
$ \neg\,  {\sf sat}_n(+,\beta,\exists v\, B)$.
 \item
  If  a sequence of the form $ \tupel{+,\beta, \forall v\, B}$ is in $X$, then
  $\neg\, {\sf sat}_n(-, \beta, \forall v\, B)$.
   \item
  If a sequence of the form $\tupel{-,\beta,\forall v\, B}$ is in $X$, then,
  for some $\gamma$ with $\beta\braee{v}\gamma$, a sequence of the form $  \tupel{-,\gamma,B}$ is in $X$.
  \end{enumerate}
  
  \noindent
  Note that if $\sigma$ and $\tau$ are extensionally equal and if
  $X$ is $n$-adequate and $\sigma$ is in $X$, then the result of replacing $\sigma$ in $X$ by $\tau$ is again
  $n$-adequate.
  
  \medent
  We will often write  $\alpha \models_n^i A$ for: ${\sf sat}_n(i, \alpha, A)$.
  The relation ${\sf sat}_n$, when restricted to $\mathcal K_n$, will have a number of 
  desirable properties.   We write ${\sf sat}^\ast_n$ for ${\sf sat}_n \cap \mathcal K_n$. 
  
  \medent
  We note that we have implicitly given a formula $\Phi_0(\mathcal X,i,\alpha, A)$, where $\mathcal X$ is a second order variable with:
  \[{\sf sat}_{n+1}(i, \alpha,A) = \Phi_0( {\sf sat}_n,i, \alpha, A).\]
  Thus, for some fixed standard $\constant{czero}$, we have 
  $\rho({\sf sat}_{n+1}(u,v,w)) = \rho({\sf sat}_n(u,v,w)) + \constantref{czero}$.
  It follows that $\rho({\sf sat}_{n}(u,v,w)) =  \constantref{czero} n +  \constant{cone}$, for some fixed standard number $\constantref{cone}$.
 
  \begin{remark}
  We note that $\mathcal X$ occurs twice in the formula $\Phi_0(\mathcal X,k,\alpha, A)$ described above. This has no 
  effect on the growth of the complexity of the formula ${\sf sat}_n$ but it makes the \emph{number of symbols}
 of the formula ${\sf sat}_n$ grow exponentially in $n$. For the purposes of this paper, this is good enough.
However, a slightly more careful rewrite of our definition reduces the number of occurrences of $\mathcal X$ to one.
As a consequence, we can get the number of symbols of ${\sf sat}_n$ linear in $n$. So, the code of ${\sf sat}_n$ will be
bounded by a polynomial in $n$, assuming we use an efficient G\"odel numbering.
  \end{remark}
  
 \noindent Our next step is to verify in ${\sf AS}^+(\Sigma)$ some good properties of $n+1$-adequacy and ${\sf sat}_n$. 
 We first show that $n$-adequacy is preserved under certain operations.
 
 \begin{theorem}\label{adeclos}
 The $n$-adequate sets are closed under unions and under intersection with the virtual class of
 $a$-sat-sequences, for any $a$. 
 \end{theorem}
 
\begin{proof}
Closure under unions is  immediate given that good sets are closed under unions.
Closure under restriction to $a$-sat-sequences is immediate by the definition of \emph{good} and the fact that
our formula classes are closed under subformulas.  
\end{proof}
 
 \noindent
We prove a theorem connecting ${\sf sat}^\ast_k$ and ${\sf sat}^\ast_n$, for $k<n$.
We remind the reader that $\mathcal K_k$ is the virtual class of all $k$-sat-sequences
and ${\sf sat}^\ast_n = {\sf sat}_n \cap \mathcal K_n$.

\begin{theorem}[${\sf AS}^+(\Theta)$]\label{babysmurf}
Suppose $k < n$. Then, ${\sf sat}^\ast_k  = {\sf sat}^\ast_n \cap \mathcal K_k$.
\end{theorem}

\begin{proof}
The proof is by external induction on $n$. The case that $n=0$ is trivial.

\medent 
Suppose $X$ is $n+1$-adequate set. Let $Y:= X\cap \mathcal K_k$.
We note that $Y$ is a set by Theorem~\ref{adeclos}.
We claim that $Y$ is $k$-adequate.  
It is clear that $Y$ satisfies all clauses for a $k$-adequate set automatically except (l) and (m).
Let's zoom in on (l).  
Suppose a sequence of the form $\tupel{-,\beta,\exists v\, B}$ is in $Y$.
From this it follows that $k\neq 0$.
Since $X$ is $n+1$-adequate, it follows that:
$ \neg\,  {\sf sat}_{n}(+,\beta,\exists v\, B)$. 
By the induction hypothesis, we find that
$ \neg\,  {\sf sat}_{k-1}(+,\beta,\exists v\, B)$. Hence, the clause for $k$-adequacy is fulfilled.
Clause (m) is similar.

\medent
Conversely, suppose $Z$ is $k$-adequate. Let $W:= Z \cap \mathcal K_k$.
The argument that $W$ is also $n+1$-adequate,
is analogous to the argument above.
\end{proof}

\noindent
We note that in the proof of Theorem~\ref{babysmurf}, we could as well do the induction on $k$.
This observation is important in case we study models with a full satisfaction predicate.
In this context, we can replace  $n$ by a non-standard number and still get our result for 
standard $k$.

In the following theorem we prove the commutation conditions for ${\sf sat}_{n+1}$. 

\begin{theorem}[${\sf AS}^+(\Theta)$]\label{commu}
We have:
\begin{enumerate}[a.]
\item
$\beta \models_{n+1}^+ P(v_0,v_1)$ iff $P(\beta(v_0),\beta(v_1))$,\\
and similarly for the other atomic formulas including $\top$ and $\bot$,
\item
$\beta \models_{n+1}^- P(v_0,v_1)$ iff $ \neg \, P(\beta(v_0),\beta(v_1))$,\\
and similarly for the other atomic formulas including $\top$ and $\bot$,
\item
  $\beta \models_{n+1}^+ \neg\, B$ iff $\beta \models_{n+1}^- B$,
\item
$\beta \models^-_{n+1} \neg\, B$ iff $\beta \models^+_{n+1} B$,
\item
  $\beta \models_{n+1}^+ B \wedge C$ iff $\beta \models_{n+1}^+ B$ and $\beta \models_{n+1}^+ C$,
  \item
  $\beta \models_{n+1}^- B \wedge C$ iff $\beta \models_{n+1}^- B$ or $\beta \models_{n+1}^- C$,
\item
  $\beta \models_{n+1}^+ B \vee C$ iff $\beta \models_{n+1}^+ B$ or $\beta \models_{n+1}^+ C$,
  \item
  $\beta \models_{n+1}^- B \vee C$ iff $\beta \models_{n+1}^- B$ and $\beta \models_{n+1}^- C$,  
  \item
  $\beta \models_{n+1}^+ B \to C$ iff $\beta \models_{n+1}^- B$ or $\beta \models_{n+1}^+ C$,
  \item
  $\beta \models_{n+1}^- B \to C$ iff $\beta \models_{n+1}^+ B$ and $\beta \models_{n+1}^- C$,  
  \item
  $\beta \models_{n+1}^+ \exists v\, B $ iff, for some $\gamma$ with $\beta\braee{v} \gamma$, we have
  $\gamma \models_{n+1}^+ B$,
  \item
  $\beta \models_{n+1}^- \exists v\, B$ iff $\beta \not \models_{n}^+  \exists v\,  B$,  
  \item
  $\beta \models_{n+1}^+ \forall v\, B$ iff $\beta \not \models_{n}^-  \forall v\,  B$, 
  \item
  $\beta \models_{n+1}^- \forall v\, B $ iff, for some $\gamma$ with $\beta\braee{v} \gamma$, we have
  $\gamma \models_{n+1}^- B$.
\end{enumerate}
\end{theorem}

\begin{proof}
We will treat the illustrative clauses (e), (k) and (l). In the first two cases the right-to-left direction is trivial.

\medent
Ad (e). Suppose $X$ witnesses that $\beta \models^+_{n+1} B$ and
$Y$ witnesses that $\beta \models ^+_{n+1} C$. Let $\sigma$ be a triple of the form
$\tupel{\beta, +,(B\wedge C)}$. Let $Z$ be a union of $X$ and $Y$ and a singleton with element
$\sigma$. It is immediate that $Z$ witnesses $\beta \models^+_{n+1} B \wedge C$.

\medent
Ad (k). Suppose that
$\beta \braee{v} \gamma$ and that $X$ witnesses that $\gamma \models_{n+1}^+ B$.
Let $\sigma$ be of the form $\tupel{\beta, +,\exists v\, B}$.
Let $Y$ be a union of $X$ and a singleton with element $\sigma$. Then $Y$ witnesses
$\beta \models^+_{n+1} \exists v\, B$.

\medent
Ad (l). Suppose  $\beta \not \models_{n}^+ B$. Let
$\sigma$ be a triple of the form $\tupel{\beta, -, \exists v\, B}$. Let $X$ be a singleton
with element $\sigma$. Then $X$ witnesses $\beta \models^-_{n+1} \exists v\, B$.
Conversely, if $Y$ witnesses $\beta \models^-_{n+1} \exists v\, B$, then we must
have $\neg \, {\sf sat}_n(\beta,+, \exists v\, B)$.
\end{proof}

\noindent 
We note that the commutation conditions are inherited by ${\sf sat}^\ast_{n+1}$, provided
that the formulas in the conditions belong to the right classes. 

\medent
 The commutation conditions proven in Theorem~\ref{commu}
 are not yet full commutation conditions. The defect is in the clauses
 (l) and (m). Let's zoom in on (m):
 \begin{itemize}
 \item
 $\beta \models_{n+1}^+ \forall v\, B$ iff $\beta \not \models_{n}^-  \forall v\,  B$.
 \end{itemize}
 The right-hand-side is equivalent to: for all $\gamma$ with $\beta\braee{v} \gamma$, we have
  $\gamma \not \models_{n+1}^- B$.  To get the desired commutation condition, we would like
  to move from $\gamma \not \models_{n+1}^- B$ to $\gamma  \models_{n+1}^+ B$.
  We have seen in Theorem~\ref{babysmurf} that to make our predicates behave in
  expected ways, it is better to consider the formulas in their `intended range'. So 
  what if $\forall v \, B$ is in $\Sigma^\ast_{n+1}$? In this case $B$ must be in $\Pi^\ast_n$
  and hence in $\Delta^\ast_{n+1}$. So is it true that if $C$ is in $\Delta^\ast_{n+1}$, 
  then $\alpha \not \models_{n+1}^- C$ iff $\alpha  \models_{n+1}^+ C$?
  To prove this we need induction, which we do  not have available in
  ${\sf AS}^+(\Theta)$. The solution is to move to a cut. 
 To realize this, we define a second measure of complexity $\nu$ (depth of connectives) as follows:
$\nu(A) :=0$ if $A$ is atomic, $\nu(\neg\, A) := \nu(\exists v\, A) := \nu(\forall v\, A) :=  \nu (A)+1$ and
 $\nu(A\circ B) := {\sf max}(\nu(A),\nu(B))+1$, where
$\circ$ is a binary propositional connective.
Let $\Gamma_x := \verz{A \mid \nu(A) \leq x}$.
We define:
\begin{itemize}
\item $J^\dag_{n+1}$ is the virtual class of all numbers $x$ such that, for all $\alpha$ and
 for all $C \in \Delta^\ast_{n+1}\cap \Gamma_x$, we have $\alpha \not \models_{n+1}^- C$ iff $\alpha  \models_{n+1}^+ C$.
\end{itemize}
We have:

\begin{theorem}[${\sf AS}^+(\Theta)$]
$J^\dag_{n+1}$ contains 0, is closed under successor and is downwards closed w.r.t. $\leq$.
\end{theorem} 

\begin{proof}
Downwards closure is immediate.

\medent
By definition, we have $\alpha \not \models_{n+1}^- C$ iff $\alpha  \models_{n+1}^+ C$ if $C$ is an atom.
So $0$ is in $J^\dag_{n+1}$. 

\medent
Suppose $x$ is in $J^\dag_{n+1}$. We will show that $x+1$ is in  $J^\dag_{n+1}$, i.e. 
 for all $C \in \Delta^\ast_{n+1}\cap \Gamma_{x+1}$, we have $\alpha \not \models_{n+1}^- C$ iff $\alpha  \models_{n+1}^+ C$.

The case for atomic $C$ follows by previous reasoning.
Suppose, for example, that $C := (D \to E)$ is  in $\Delta^\ast_{n+1} \cap \Gamma_{x+1}$.
Then $D$ and $E$ are both in $\Delta^\ast_{n+1} \cap \Gamma_{x}$.
By the fact that $x$ is in $J^\dag_{n+1}$ we find:
\begin{eqnarray*}
 \alpha \models^+_{n+1} (D \to E)  & \iff & (\alpha \models^-_{n+1} D) \text{ or } ( \alpha\models_{n+1}^+ E) \\
 & \iff & (\alpha \not \models^+_{n+1}D)\text{ or } (\alpha\not \models_{n+1}^- E) \\
 & \iff & \neg\, ((\alpha \models^+_{n+1} D) \text{ and } (\alpha \models_{n+1}^- E)) \\
 & \iff & \neg\; \alpha \models_{n+1}^- (D \to E).
 \end{eqnarray*}
 The other unary and binary propositional connectives are similar. Now suppose $C$ is of the form $\exists v\, D$ and 
 $C \in  \Delta^\ast_{n+1}\cap \Gamma_{x+1}$.
 Since $C$ is in $\Pi^\ast_{n+1}$, we must have $D \in \Sigma^\ast_n$. It follows that $n \neq 0$. We have, by
 Theorem~\ref{babysmurf}:
 \begin{eqnarray*}
 \alpha \models^+_{n+1} \exists v\, D  & \iff &  \alpha \models^+_{n} \exists v\, D\\
 & \iff & \alpha \not\models^-_{n+1} \exists v\, D 
 \end{eqnarray*}
 Curiously, this step does not not use the fact that $x$ is in $J^\dag_{n+1}$. The case of $\forall$ is similar.
\end{proof}

\noindent
Thus, for any $C$ in $\Delta_{n+1}^\ast \cap \Gamma_{J^\dag_{n+1}}$
we have  $\alpha \not \models_{n+1}^- C$ iff $\alpha  \models_{n+1}^+ C$.
Hence, we also have the full Tarskian commutation clauses for these $C$.

We note that $J_{n+1}^\dag$ is $\Phi_1(n+1,{\sf sat}_{n+1},x)$, for a standard formula $\Phi_1(y,\mathcal X,x)$. 
Thus the $\rho$-complexity of $J_{n+1}^\dag$ is linear in $n$ where the relevant linear term is of the form
$\constantref{czero} n+ \constant{cthree}$. 

\medent
We still miss an important ingredient. 
Let $\tupel{i,\alpha,A}$ be a sat-sequence.
Suppose $\alpha$ and $\beta$ assign the same values to the free variables in
$A$. Do we have $\alpha \models^i_n A$ iff $\beta \models^i_n A$?
To prove such a thing we need induction. We would like to have even more than this,
since we want to check the validity of the inference rules for the quantifiers.
We define the property ${\sf Q}_{n+1}$ by:
\begin{itemize}
\item
The formula $A$ has the property ${\sf Q}_{n+1}$ if the following holds.
Consider any sat-sequence $\tupel{i,\alpha,A}$.
Suppose $w$ is free for $v$ in $A$. Let $B$ be (of the form) $A[v:= w]$.
Suppose further that the functions
 $\alpha[u]=\beta[u]$ for all free variables $u$ of $A$, except possibly $v$, and that $\beta[w]=\alpha[v]$. 
Then,  $\alpha \models^i_{n+1} A$ iff $\beta \models^i_{n+1} B$.
\end{itemize}
We allow that $v$ and $w$ are equal and that $v$ does not occur in $A$. Both degenerate cases tell us that ${\sf Q}_{n+1}(A)$ implies that
$\alpha \models^i_{n+1} A$ iff $\beta \models^i_{n+1} A$, whenever $\alpha$ and $\beta$
agree on the free variables of $A$.

 We can now proceed in two ways to construct a cut that gives us the desired property for the formulas
 of $\nu$-complexity in the cut. 
 One way does not involve the $\Sigma_n^\ast$ and the $\Pi_n^\ast$ and
 one way  does involve them. The second way yields a more efficient construction of the cut.
 For completeness, we explore both ways.
 
 We first address the first way.
 We define:
 \begin{itemize}
 \item
 $J^\circ_0 := {\sf N}$, $x \in J^\circ_{n+1} : = \verz{x \in J^\circ_n \mid  \Gamma_x \subseteq {\sf Q}_{n+1}}$.
 \end{itemize}
 
 \noindent
 We note that the definition of $J^\circ_{n+1}$ is of the form $\Phi_2({\sf sat}_{n+1},J^\circ_n,x)$ for a fixed
 $\Phi_2(\mathcal X, \mathcal Y,x)$. So,
 $\rho(J^\circ_{n+1}) = {\sf max}({\rho(\sf sat}_{n+1}),\rho(J^\circ_n))+\constant{cfour}$, for a fixed standard $\constantref{cfour}$.
 It follows that $\rho(J^\circ_{n})$ is estimated some linear term $\constant{cfive} n + \constant{csix}$.

  \begin{theorem}[${\sf AS}^+(\Theta)$]\label{tuinsmurf}
  The virtual class $J^\circ_n$ is closed under $0$, successor, and is downwardly closed w.r.t. $\leq$.
\end{theorem}

\begin{proof}
Closure under 0 and downwards closure are  trivial.
We prove closure under successor by induction on $n$. The case of $J^\circ_0$ is trivial. 
Suppose  that $J^\circ_n$ is closed under successor.  Consider $x$ in $J^\circ_{n+1}$.
Let  $C$ and $D$ be in $\Gamma_x $.

\medent
Let
 $A$ be of the form $(C\wedge D)$.
Suppose $w$ is free for $v$ in $A$. Let $B$ be (of the form) $A[v:= w]$.
Suppose further that
 $\alpha$ and $\beta$ assign the same values to the free variables of
 $A$ except $v$ and $\beta[w]=\alpha[v]$. Clearly $B$ is of the form $E \wedge F$, where
 $E$ is of the form $C[v:=w]$ and $F$ is of the form $D[v:=w]$.
 We have:
 \begin{eqnarray*}
\alpha \models^{+}_{n+1} C \wedge D & \iff & \alpha \models^{+}_{n+1} C \text{ and } \alpha \models^{+}_{n+1} D   \\
& \iff &  \beta \models^{+}_{n+1} E \text{ and } \beta  \models^{+}_{n+1} F \\
& \iff &  \beta \models^{+}_{n+1} E \wedge F
\end{eqnarray*}
Similarly for the $\models^-$-case. The other propositional cases are similar.

\medent 
We treat the case of the existential quantifier, the case of the universal quantifier being similar.
Let $A$ be of the form  $\exists z\, C$.
Suppose $w$ is free for $v$ in $A$. Let $B$ be (of the form) $A[v:= w]$.
Suppose further that
 $\alpha$ and $\beta$ assign the same values to the free variables of
 $A$ except $v$ and $\beta[w]=\alpha[v]$. 
 
 We first address the $\models^+$-case.
 The argument splits into two subcases. First we have the
  case that  $z$ is (of the form) $v$, we find that
 $B$ is of the form $A$. Hence, replacing $z$ by $v$, we have:
  \begin{eqnarray*}
\alpha \models^{+}_{n+1} \exists v\, C & \iff & \exists \gamma\, ( \alpha\braee{v}\gamma \text{ and } \gamma \models^{+}_{n+1} C )   \\
& \iff &  \exists \delta\, ( \beta \braee{v}\delta \text{ and } \delta \models^{+}_{n+1} C )   \\
& \iff & \beta \models^{+}_{n+1} \exists v\, C 
\end{eqnarray*}
 
 \noindent
For example, in the left-to-right direction of the second step we can take $\delta$ of the form $\beta[v: \gamma(v)]$. 
We can use the fact that $C$ has $\nu$-complexity $x$ and $x  \in J^\circ_{n+1}$. The property ${\sf Q}_{n+1}$ is applied
 with $v$ in the role of both $v$ and $w$.

Next we have the case that $z$ and $w$ are different variables. Let $D$ be of the form $C[v:w]$.
So $B$ is of the form $\exists z\, D$. We have:
 \begin{eqnarray*}
\alpha \models^{+}_{n+1} \exists z\, C & \iff & \exists \gamma\, ( \alpha\braee{z}\gamma \text{ and } \gamma \models^{+}_{n+1} C )   \\
& \iff &  \exists \delta\, ( \beta \braee{z}\delta \text{ and } \delta \models^{+}_{n+1} D )   \\
& \iff & \beta \models^{+}_{n+1} \exists z\, D 
\end{eqnarray*}
E.g., in the left-to-right direction of the second step we can again take $\delta$ of the form $\beta[z: \gamma(z)]$. 

Finally we address the $\models^-$-case. We have:
  \begin{eqnarray*}
\alpha \models^{-}_{n+1} \exists v\, C & \iff & \neg\;\; \alpha \models^+_{n}  \exists v\, C  \\
& \iff & \neg\;\; \beta \models^+_{n}  \exists v\, D  \\
& \iff &  \beta \models^-_{n+1}  \exists v\, D   
\end{eqnarray*}
 
\noindent Here we use the fact that $x \in J^\circ_n$, so that also $x+1 \in J_n^\circ$. 
\end{proof}

\noindent
We turn to the second approach. 
We define:
\begin{itemize}
\item
$J^\star_{n+1} := \verz{x\in {\sf N} \mid ( \Gamma_x \cap \Delta^\ast_{n+1}) \subseteq {\sf Q}_{n+1}}$.
\end{itemize}
We have:

\begin{theorem}[${\sf AS}^+(\Theta)$]
The virtual class $J^\star_{n+1}$ is closed under 0, successor and is downwards closed w.r.t. $\leq$.
\end{theorem}

\begin{proof}
The cases of closure under 0 and downwards closure are trivial. 
Suppose $x$ is in $J^\star_{n+1}$. 
The cases of the propositional connectives use
the same argument as we saw in the proof of theorem~\ref{tuinsmurf}. We turn to the case of the existential quantifier, the
case of the universal quantifier being dual. Suppose $C$ is in $\Gamma_x \cap \Delta^\ast_{n+1}$. The case of $\models^+$ is
again the same as we saw in the proof of Theorem~\ref{tuinsmurf}. We consider the case of $\models^-$.
Suppose $\exists v\, C$ is in $\Delta^\ast_{n+1}$. In this case $\exists v\, C$ must be in $\Sigma^\ast_n$.
We have:
  \begin{eqnarray*}
\alpha \models^{-}_{n+1} \exists v\, C & \iff & \neg\, \alpha \models^+_{n}  \exists v\, C  \\
& \iff & \neg\, \alpha \models^+_{n+1}  \exists v\, C  \\
& \iff & \neg\, \beta \models^+_{n+1}  \exists v\, D  \\
& \iff & \neg\, \beta \models^+_{n}  \exists v\, D  \\
& \iff &  \beta \models^-_{n+1}  \exists v\, D   
\end{eqnarray*}
 The first and the fifth step use the commutation conditions for $\exists$. The second and the fourth step use 
 Theorem~\ref{babysmurf}. The third step uses the previous case for $\models^+$.
\end{proof}

\noindent
We note that the definition of $J_{n+1}^\star$ is of the form $\Phi_3(n,{\sf Sat})$. So, its $\rho_0$-complexity
 is estimated by a linear term of the form $\constantref{czero} n + \constant{cseven}$.
Here the use of $J_{n+1}^\star$ has an advantage over $J_n^\circ$, since construction of
the $J_m^\circ$ gives us a linear complexity but conceivably with a higher constant as coefficient of $n$. 

\medent
Let us take stock of what we accomplished. 
We have defined virtual classes $J_{n+1}^\star$ that are closed under $0$ and {\sf S} and that are downwards closed
such that for all formulas $A$ in $\Gamma_{J_{n+1}^\star} \cap \Delta^\ast_{n+1}$, we have, for all $\alpha$, that
$\alpha \models_{n+1}^+ A$ iff $\alpha \not \models_{n+1}^- A$. Here
$\Gamma_{J_{n+1}^\star}   := \bigcup_{x\in J_{n+1}^\star} \Gamma_x$.

Also, we have developed
virtual classes $J_{n+1}^\circ$ and $J_{n+1}^\star$ such that all $A$ in $\Gamma_{J_{n+1}^\circ}$, and, similarly, all  
$A$ in $\Gamma_{J_{n+1}^\star} \cap \Delta_{n+1}$
have the property
${\sf Q}_{n+1}$ defined above. 

So, if we take $J_{n+1}^\ddag$ either $J_{n+1}^\star \cap J_{n+1}^\circ$ or $J_{n+1}^\star \cap J_{n+1}^\star$
then $J_{n+1}^\ddag$ is progressive and all elements of $\Xi_{n+1} := \Gamma_{J_{n+1}^\ddag} \cap \Delta_{n+1}$ have both good properties.
Let us choose for $J^\star$ in the definition of $\Xi_{n+1}$, so that its $\rho$-complexity is estimated by
$\constantref{czero}n + \constant{ceight}$.

We summarize the result in a theorem.

\begin{theorem}\label{boeksmurf}
We have full commutation of ${\sf sat}(+,\cdot,\cdot )$ for the $\Xi_{n+1}$-formulas. Moreover, the
$\Xi_{n+1}$-formulas have propertu ${\sf Q}_{n+1}$.
\end{theorem}
  
\noindent
  We write $\mathscr A \rightslice_n A$ for the formalization of $\mathscr A \vdash_n A$, where $\mathscr A$ codes a finite set of formulas and
   $\vdash_n$ is 
  provability in predicate logic where we restrict ourselves in the proof to $\Xi_{n+1}$-formulas.
  We choose to code the set of formulas in the natural numbers. This is a bit unnatural since ${\sf AS}(\Theta)$
  contains sets as first-class citizens. However, if we code sets of formulas in the sets provided by ${\sf AS}(\Theta)$
  directly we do not know, for example, that ${\sf ass}(p)$ the set of assumptions of a proof $p$ is a set.
  Of course, this problem can be evaded by shortening {\sf N} in such a way that any set coded in the natural numbers
  maps to first-class set. If the reader prefers this other road, we think it is sufficiently clear how to adapt the
  results below to this alternative approach.
  
   We write $p:\mathscr A \rightslice_n A$ for: $p$ is the code of a $n$-proof witnessing $\mathscr A  \rightslice_n A$.
  We write $\Lambda_{n,y}$ for the class of $n$-proofs $p$  where the number of steps of $p$ is $\leq y$.
On the semantical side,  we define, for $\mathscr A \cup \verz{A} \subseteq  \Xi_{n+1}$:
  \[ \mathscr A \models_{n+1} A : \iff 
  \forall \alpha\,
(\forall A'\in \mathscr A\, {\sf sat}_{n+1}(+,\alpha,A') \to {\sf sat}_{n+1}(+,\alpha,A)). \]
  
  \noindent
  We work in ${\sf AS}^+(\Theta)$. We write ${\sf ass}(p)$ for the assumption set of (proof code) $p$.
  Let $Y_n$ be the class of $y$ such that,
  for all $p\in \Lambda_{n,y}$, if $p:{\sf ass}(p) \rightslice_n A$, then
  ${\sf ass}(p)\models_n A$.
  
  \begin{theorem}
  The virtual class $Y_n$ is downwards closed under $\leq$, contains 0, and is closed under successor.
  \end{theorem}
  
  \begin{proof}
  Downwards closure under $\leq$ is trivial. We show that $Y_n$ is progressive.  
We follow the system for Natural Deduction in sequent style as given in \cite[Subsection 2.1.4]{troe:basi00}. 
By Theorem~\ref{boeksmurf}, the propositional cases are immediate.
We will treat the introduction and the elimination rule of the universal quantifier. This follows mainly the usual text book
proof. For the convenience of the reader, we repeat the property ${\sf Q}_{n+1}$:
\begin{itemize}
\item
The formula $C$ has the property ${\sf Q}_{n+1}$ if the following holds.
Consider any sat-sequence $\tupel{i,\alpha,C}$.
Suppose $u$ is free for $z$ in $C$. 
Suppose further that
 $\alpha[a]=\beta[a]$ for all free variables $a$ of
 $A$ except $z$ and that $\beta[u]=\alpha[z]$. 
Then,  $\alpha \models^i_{n+1} C$ iff $\beta \models^i_{n+1} C[z:=u]$.
\end{itemize}

\noindent
We treat the case of universal generalization. 
Let $w$ be substitutable for $v$ in $A$ and suppose
$w$ does not occur freely in the elements of $\mathscr A\cup\verz{A}$.
Suppose we have (\dag) $\mathscr A \models_{n+1} A[v:=w]$. We show that $\mathscr A \models_{n+1} \forall v\, A$.

Consider any $\alpha$ and suppose $\alpha \models_{n+1}^+ A'$, for all $A'$ in $\mathscr A$. We want to show that
 $\alpha \models_{n+1}^+ \forall v\, A$.  Let $d$ be any element. It is clearly sufficient to show that $\alpha[v:d] \models_{n+1}^+ A$.

We first note that $\alpha[w:d]$ and $\alpha$ are the same on the free variables of $A'$  in $\mathscr A$.  
So, by ${\sf Q}_{n+1}$ in a degenerate case, we find that
$\alpha[w:d] \models^+_{n+1}A'$, for all $A'$ in $\mathscr A$. By (\dag), we find $\alpha[w:d] \models^+_{n+1}A[v:= w]$.
We note that  $\alpha[w:d]$ and $\alpha[v:d]$ assign the same values to all free variables $u$ of $A$, except possibly $v$.
This uses that $w$ does not occur in $A$. Moreover, $\alpha[w:d][w] = \alpha[v:d][v]$. So, by ${\sf Q}_{n+1}(A)$, 
we find $\alpha[v:d] \models_{n+1}^+ A$ iff $\alpha[w:d] \models A[v:= w]$.  Thus, we may conclude $\alpha[v:d] \models_{n+1}^+ A$
as desired.

\medent
We treat the case of universal instantiation.
Suppose (\ddag) $\mathscr A \models \forall v\, A$. We want to conclude $\mathscr A \models A[v:=w]$.
Suppose, for all $A'\in \mathscr A$, we have $\alpha \models^+_{n+1} A'$.
It follows that (\$) $\alpha \models^+_{n+1} \forall v\, A$. We want to conclude that $\alpha \models A[v:= w]$.
From (\$), we have $ \alpha[v:= \alpha[w]] \models^+ A$. We note that $\alpha$ and $ \alpha[v:= \alpha[w]]$
assign the same values to all free variables of $A$ except possibly $v$. Moreover $ \alpha[v:= \alpha[w]][v] = \alpha[w]$. 
By ${\sf Q}_{n+1}(A)$, we may conclude that $\alpha \models^+_{n+1} A[v:= w]$.
 \end{proof} 
 
 \noindent
  Inspecting the construction of $Y_n$ we see that it is of the form $\Phi_4(\underline n, {\sf sat}_n,J^\ddag_{n+1})$, where
  $\Phi_4(x,\mathcal X,\mathcal Y)$ is a fixed formula. Thus $\rho(Y_n)$ is estimated by $\constantref{czero}n + \constant{cnine}$.
  
  \medent
  We now have a refined result involving separate restrictions on $\rho$ on $\nu$ and on the length of the proofs. 
  For other applications this refinement may be useful, however, in the present paper, we will simply demand
  that our proofs are in a cut $\Im_n(\Theta)$ that is obtained by taking the intersection of $J^\ddag_{n+1}$ and
  $Y_n$ and shortening to obtain downwards closure and closure under $0$, {\sf S}, $+$, $\times$ and $\omega_1$.
  Since the shortening procedure only adds a standardly finite depth to the input formula $J^\ddag_{n+1}\cap Y_n$,
  $\rho(\Im_n(\Theta))$ will have complexity  $\constantref{czero}n + \constant{cten}$. Moreover, when $p$ is in $\Im_n(\Theta)$, then ipso facto
  its length is in $Y_n$ and its $\nu$-complexity is in $J^\ddag_{n+1}$. 
  
  We write $ [ \mathscr A \vdash_n A]$ for provability in predicate logic involving only $\Delta_n^\ast$-formulas where the proof is constrained to be
  in the cut $J$. We write $ [ \mathscr A \vdash^J_n A]$ when the witness for $ [ \mathscr A \vdash_n A]$ is constrained to the cut $J$.
  We write   $\opr_{\Theta,n}A$ for $[ \emptyset \vdash_n A]$, and $\opr^J_{\Theta,n}A$ for $[ \emptyset \vdash^J_n A]$.
  For sentences $A$, we will write ${\sf true}_{\Theta,n} (A)$ for $\forall \alpha\, {\sf sat}_n(+,\alpha,A)$.
  
  \begin{theorem}\label{summumsmurf}
  We can find an $\omega_1$-cut $\Im_n(\Theta)$ such that $\rho(\Im_n(\Theta))$ is of order $\constantref{czero}n + \constantref{cten}$ and
  such that:
  \[ {\sf AS}^+(\Theta) \vdash \forall  \mathscr A , A  \, ([\mathscr A \vdash^{\Im_n(\Theta)}_n A] \to  \mathscr A \models_n A).\]
  As a special case, we have:
 \[{\sf AS}^+(\Theta) \vdash \forall A \in {\sf sent}^{\Im_n(\Theta)} \, (\opr_{\Theta,n}^{\Im_n(\Theta)}A \to {\sf true}_{\Theta,n}(A)).\]
  \end{theorem}
  
\section{Small-Is-Very-Small Principles}\label{rost}
In this section we present the central argument of this paper. It is a simple Rosser argument. 
The bulk of the work has already been done in creating the setting for the result. I choose to give the pure argument
in Theorem~\ref{grotesmurf} rather than proceed immediately to the somewhat more
complicated Theorem~\ref{grotesmurfdc}. The more complicated version is needed for
application in model theory. 

\medent
First some preliminaries and notations, in order to avoid too heavy notational machinery.

We will work in sequential theories $U$ of signature $\Theta$ with sequence scheme $\mathcal S$. 
So, $\mathcal S: {\sf AS}^+ \stackarrow{{\sf dir}} U$. We can lift $\mathcal S$ to a direct interpretation 
$\mathcal S_\Theta:{\sf AS}^+(\Theta)\stackarrow{\sf dir} U$ by translating $\Theta$ identically.\footnote{In case $\mathcal S$ would be a sequence scheme for
a polysequential we would need a slight adaptation.}

We remind the reader that $\mathscr N:{\sf S}^1_2 \to {\sf AS}^+$.
We will write $N := \mathcal S \circ \mathscr N :{\sf S}^1_2 \to U$. So, e.g., $\delta_N = {\sf N}^{\mathcal S}$.
We write $\Im_n$ for $(\Im_n(\Theta))^{\mathcal S_\Theta}$. When we write numerals $\underline n$ these
are always numerals w.r.t. $N$. We note that the numerals really are eliminated using the term elimination algorithm.
However, this elimination just gives an overhead on 1 in $\rho_0$-complexity.

Let $\eta$ be a $\Sigma_1^{\sf b}$-formula defining a set of axioms. We write $\opr_\eta$ for provability from
the axioms in $\eta$. We write $\opr_\eta^J$ for the result of restricting the witnesses for $\opr_\eta$ to $J$.
We write  $\opr_{\eta,n}$ for the result of restricting the formulas in a witnessing proof to $\Delta_n^\ast$.\footnote{In previous papers, I
also used this notation to signal that the codes of the axioms were constrained to be $\leq n$. In this paper
 this extra demand is not made.} Formulas like 
$\opr^J_{\eta,n}$ have the obvious meanings.
We suppress the information about the signature $\Theta$, which should be clear from the context.
In case $\eta = (x = \gnum{A})$, we write $\opr_A$ for $\opr_\eta$.\footnote{Clearly, this introduces an ambiguity. E.g., does $\opr_\top$ mean 
provability form all sentences or from the axiom $\top$? However, what we intend will be always clear from the context,}

We will employ witness comparison notation:
\begin{itemize}
\item
$\exists x\in \delta_n\, A_0(x) \leq \exists y\in \delta_N\, B_0(y)$ iff $\exists x\in \delta_N\, (A_0(x) \wedge \forall y <^N x\, \neg\, B_0(y))$.
\item
$\exists x\in \delta_n\, A_0(x) < \exists y\in \delta_N\, B_0(y)$ iff $\exists x\in \delta_N\, (A_0(x) \wedge \forall y \leq^N x\, \neg\, B_0(y))$.
\item
$(\exists x\, A_0(x) \leq \exists y\, B_0(y))^\bot  := \exists y\, B_0(y) < \exists x\, A_0(x)$.
\item
$(\exists x\, A_0(x) < \exists y\, B_0(y))^\bot  := \exists y\, B_0(y) \leq \exists x\, A_0(x)$.
\end{itemize}

\begin{theorem}\label{grotesmurf}
Let $A$ be a finitely axiomatized sequential theory in a language with signature $\Theta$
 with sequence scheme $\mathcal S$.
 Consider any sentence $B$ in the language of $A$  of the form   $B := \exists x\in \delta_N \, B_0(x)$.
 Let $n := {\sf max}(\rho_0(A),\rho_0(B)+\constant{celeven},\rho_0(\mathcal S) +\constantref{celeven})$. Here $\constantref{celeven}$ is a fixed finite constant that does not
 depend on $A$, $B$ and $\mathcal S$. We will determine $\constantref{celeven}$ below.

Suppose   $A \vdash \exists x\in \Im_n\, B_0(x)$.
 Then, for some $k$, we have $A \vdash \exists x \leq^N \underline k\, B_0(x)$, or, equivalently,
$A \vdash \bigvee_{q\leq k}\, B_0(\underline q)$. 
\end{theorem}

\begin{proof}
We work under the conditions of the theorem.
Using the G\"odel Fixed Point Lemma, we find a sentence $R$
such that $A \vdash R \iff B \leq {\opr_{A,\underline n}^N R}$. 
We need that $\rho_0(R) \leq n$. 

\medent
{\small Under the usual G\"odel construction, $R$ is of the following form:
\[ \exists x\, ( \delta_N(x) \wedge B_0(x) \wedge 
\exists z\, (\delta_N(z) \wedge {\sf sub}^N(\underline \ell, \underline \ell, z) \wedge
 \forall y\, (y <^N x \to \neg\, {\sf prov}_{A,\underline n}^N(y,z)))).\] 
 Thus, $\rho_0(R)$ is estimated by \[ {\sf max}(\rho_0({\sf sub}),\rho_0({\sf prov})) + {\sf max}(\rho_0(\mathcal S), \rho_0(B_0)) +3.\]
 Here the $+3$ is due to the additional quantifiers. We note that, if we unravel the numerals $\underline n$ wide scope, we
 even just need $+2$.
 So, we can take $\constantref{celeven} :=  {\sf max}(\rho_0({\sf sub}),\rho_0({\sf prov})) +3$.}

\medent
Reason in $A$. We have $\exists x \in \Im_n\, B_0(x)$.  In case $\neg\, \opr^{\Im_n}_{A,n} R$, we have $R$.
Suppose $\opr^{\Im_n}_{A,n} R$. By reflection, as guaranteed by Theorem~\ref{summumsmurf}, we find $R$.
So, in both cases, we may conclude that $R$. We leave $A$ again.

\medent
Thus, we have shown (i) $A \vdash R$. By cut-elimination, we find:
$A \vdash_n R$. 
 Hence, (ii) for some $k$, we find
$A \vdash {\sf proof}^N_{A,n}(\underline k,  R)$.
Combining (i) and (ii), we  we may conclude that $A \vdash \exists x \leq^N \underline k\, B_0(x)$, or, equivalently,
$A \vdash \bigvee_{q\leq k}\, B_0(\underline q)$. 
\end{proof}

\noindent
We note that, due to the use of cut-elimination, we need the totality of superexponentiation in the metatheory.
Such theorems usually leave watered-down traces in weaker metatheories. We do not explore such possibilities in the present paper.

The above argument has some analogies with Harvey Friedman's beautiful proof that, in a constructive setting, the disjunction property
implies the existence property. See \cite{frie:disj75}. I analyzed this argument in \cite{viss:fefe14}, having the benefit of many perceptive
remarks by Emil Je\v{r}\'abek. One surprising aspect of the above proof is that the minimization principle is not used. Joost Joosten pointed
out to me in conversation that the closely related Friedman-Goldfarb-Harrington Theorem also can be proven without using minimization. 

For our model theoretic applications we need a variant of Theorem~\ref{grotesmurf} that adds domain constants.
We allow for the domain constants the exceptional position that they are real constants rather than
unary predicates posing as constants.

\begin{theorem}\label{grotesmurfdc}
Consider a finite set of domain constants $\mathcal C$. 
Let $A_0$ be any finitely axiomatized sequential theory with signature $\Theta$ and sequence scheme $\mathcal S$.
Let $A_1:=A_1(\vec c\,)$ be any sentence in the language with signature $\Theta+\mathcal C$. Let $A := A_0 \wedge A_1$.
 
 Consider any sentence $B(\vec c\,)$ in the language of signature $\Theta+\mathcal C$  of the form   $B(\vec c\,)  := \exists x\in \delta_N\, B_0(x,\vec c\,)$. 
 Let $n := {\sf max}(\rho_0(A),\rho_0(B)+\constantref{celeven},\rho_0(\mathcal S) +\constantref{celeven})$. 

Suppose $A(\vec c\,) \vdash \exists x\in \Im_n \, B_0(x,\vec c\,)$.\footnote{The fact that the constants in  $\mathcal C$
do not occur in  $\Im_n = \Im^{\mathcal S}_n(\Theta)$ is the whole point of the refined result.} 
Then,
for some $k$, we have that  $A(\vec c\,) \vdash \exists x \leq^N \underline k\, B_0(x,\vec c\,)$, or, equivalently,
$A(\vec c\,) \vdash \bigvee_{q\leq k}\, B_0(\underline q,\vec c\,)$. 
\end{theorem}

\begin{proof}
We work under the conditions of the theorem. 
We find a sentence $R(\vec c\,)$
such that $A(\vec c\,) \vdash R(\vec c\,) \iff B(\vec c\,) \leq {\opr_{A(\vec c\,),\underline n}^N R(\vec c\,)}$. 

\medent
Reason in $A(\vec c)$.  In case $\neg\, \opr^{\Im_n}_{A(\vec c\,),n} R(\vec c\,)$, we have $R(\vec c\,)$.
Suppose $\opr^{\Im_n}_{A(\vec c\,),n} R(\vec c\,)$.
Replacing the extra constants $\vec c$, in $A(\vec c\,)$ and $R(\vec c\,)$ by fresh
variables $\vec v$, we get $[A(\vec v\,) \vdash^{\Im_n}_n R(\vec v\,)]$. Hence, $A(\vec v\,) \models_n R(\vec v\,)$.
Since $A$ and $R$ are standard and since we have $A(\vec c\,)$, we find $R(\vec c\,)$. 

Thus, we have shown (i) $A(\vec c\,) \vdash R(\vec c\,)$. By cut-elimination, we find:
$A(\vec c\,) \vdash_n R(\vec c\,)$. 
 Hence, (ii) for some $k$, we have 
$A(\vec c\,) \vdash {\sf proof}^N_{A(\vec c\,),n}(\underline k,  R(\vec c\,))$.
Combining (i) and (ii), we get $A(\vec c\,) \vdash \exists x \leq^N \underline k\, B_0(x,\vec c\,)$, or, equivalently,
$A(\vec c\,) \vdash \bigvee_{q\leq k}\, B_0(\underline q,\vec c\,)$. 
\end{proof}

\noindent
We call a theory $U$ \emph{restricted} if, for some $m$ all its axioms are in $\Delta^\ast_m$.

\begin{theorem}\label{grotesmurfres}
Suppose $A_0$ is a finitely axiomatized sequential theory in signature $\Theta$ with sequence
scheme $\mathcal S$.
Let $m$ be any number such that $m \geq \rho_0(A_0)$. Let $\mathcal C$ be a set of domain constants:
$\mathcal C$ is allowed to have any cardinality. Let $U$ be a restricted theory bounded by $m$ in the
language of signature $\Theta+\mathcal C$ extending $A_0$. The theory $U$ may have any complexity.

Consider any sentence $B$ in the language of signature $\Theta+\mathcal C$  of the form   $B := \exists x\in \delta_N\, B_0(x)$.
 Let $n := {\sf max}(m,\rho_0(B)+\constantref{celeven},\rho_0(\mathcal S) +\constantref{celeven})$. 
 
 Suppose $U \vdash \exists x\in \Im_n \, B_0(x)$.
Then,
for some $k$, we have $U \vdash \exists x \leq^N \underline k\, B_0(x)$, or, equivalently,
$U \vdash \bigvee_{q\leq k}\, B_0(\underline q)$. 
\end{theorem}

\begin{proof}
The theorem is immediate from Theorem~\ref{grotesmurfdc}, using compactness.
\end{proof}

\noindent
It is of course trivial to take the contraposition of Theorem~\ref{grotesmurfres}. However this contraposion
has some heuristic value. So we state it here as a separate theorem.

\begin{theorem}\label{ijdelesmurf}
Suppose $A_0$ is a finitely axiomatized sequential theory in signature $\Theta$ with sequence
scheme $\mathcal S$.
Let $m$ be any number such that $m \geq \rho_0(A_0)$. Let $\mathcal C$ be a set of domain constants:
$\mathcal C$ is allowed to have any cardinality. Let $U$ be a restricted theory bounded by $m$ in the
language of signature $\Theta+\mathcal C$ extending $A_0$. The theory $U$ may have any complexity.

Consider any formula $C(x)$. Let $n: ={\sf max}(m,\rho_0(C)+\constantref{celeven},\rho_0(\mathcal S) +\constantref{celeven})$. 

If the theory $U+ \verz{C(\underline q) \mid q \in \omega}$ is consistent, then the theory
$U + \forall x \in \Im_n\, C(x)$ is consistent.
\end{theorem}

\begin{proof}
We apply Theorem~\ref{grotesmurfres} to $\exists x\in \delta_N\, \neg \, C(x)$ and  take the contraposition.
\end{proof}

\section{A Conservativity Result}\label{core}
We can use the machinery we built up to prove a Lindstr\"om-style result on conservative extensions.

Suppose $U$ is a restricted, sequential, recursively enumerable theory with sequence scheme $\mathcal S$.
  Let $p$ be a bound for the complexity of the
axioms of $U$. 
By Craig's trick,  we can give a $\Sigma_1^{\sf b}$-axiomatization of $U$. Say the   $\Sigma_1^{\sf b}$-formula representing the
axioms is $\eta$. Suppose $A_0$ is a finite subtheory of $U$ such that $\mathcal S$ makes $A_0$ sequential.
Clearly $U$ can be axiomatized by \[ A_0 + \verz{ \eta^N(\underline q) \to {\sf true}_n(\underline q) \mid q \in \omega}.\]

\noindent
This representation of the axiom set leads immediately to the following theorem.

\begin{theorem}\label{conservatievesmurf}
Suppose $U$ is a restricted, sequential, recursively enumerable theory. Consider any number $m$.
Then there is a finitely axiomatized sequential theory $A$ in the same language that extends $U$ and is $\Delta_m^\ast$-conservative over $U$.
\end{theorem}

\begin{proof}
By our above observations there is a finitely axiomatized sequential theory
$A_0$ and a formula $B(x)$ such that $U$ can be axiomatized as $A_0+\verz{B(\underline q) \mid q\in \omega}$.
 Let $n := {\sf max}(\rho_0(A_0), \rho_0(B)+\constantref{celeven}, m+\constantref{celeven}, \rho_0(\mathcal S)+\constantref{celeven})$. 

We take $A := A_0 +\forall x\in \Im_n \, B(x) $. We note that $A$ is a finitely axiomatized extension of
$U$. Consider any $C \in \Delta_m^\ast$. Suppose $U \nvdash C$. Then, the theory
$A_0 + \verz{ (B(\underline q) \wedge \neg \, C) \mid q\in \omega}$ is consistent. It follows that 
 $A_0 +\forall x\in \Im_n \, (B(x) \wedge \neg\, C)$ is consistent. In other words, we find $A \nvdash C$.  
\end{proof}

\noindent
We have the following corollary.
\begin{corollary}
Suppose $U$ is a restricted, sequential, recursively enumerable theory. Suppose further that $D$ is a finite extension of $U$ such that $U \nvdash D$.
Then, there is a finite extension $D'$ of $U$, such that $D \vdash D'$ but $D'\nvdash D$. 
\end{corollary}

\begin{proof}
Let $m$ be a $\rho_0$-bound on $U$ and on $D$.  Let $A$ be the sentence promised in Theorem~\ref{conservatievesmurf} for
$\Delta_m^\ast$. Let $D' := D \vee A$. Clearly $D \vdash D'$ and $D' \vdash U$. Suppose
$D' \vdash D$. Then, it follows that $A \vdash D$, contradicting the $\Delta^\ast_m$-conservativity of
$A$ over $U$.
\end{proof}
 
\noindent
Since, as is well-known, the finitely axiomatized sequential theories in the signature of $U$ are dense
w.r.t. $\dashv$, it follows that we can add to the statement of the Corollary that $U \nvdash D'$: in case the
$D'$ provided the theorem would happen to axiomatize $U$, we simply replace it by a $D''$ strictly between
the original $D'$ and $D$. 
  
Here is one more corollary.

\begin{corollary}\label{tarskismurf}
Consider any finitely axiomatized, sequential theory $A$ in signature $\Theta$. Suppose that
for some of class of $\Theta$-sentences $\Omega$ we have a definable predicate {\sf TRUE}
such that, for any $\Omega$-sentence $B$, we have $A \vdash B \iff {\sf TRUE}(\gnum B)$. 
Let $X$ be any recursively enumerable set of $\Omega$-sentences. Then there is a finite extension $A^+$ of $A+X$ such that
$A^+$ is $\Omega$-conservative over $A+X$.
\end{corollary}

\begin{proof}
Clearly $A + \verz{{\sf TRUE}(B) \mid B \in X}$ is restricted. We apply Theorem~\ref{conservatievesmurf}
taking $m := \rho_0({\sf TRUE}(x)) +1$. 
\end{proof}

\noindent  We give two examples of applications of the result.

\begin{example}
Let $U$ be any recursively enumerable extension of {\sf PA}. Then, there is a finite extension $A$ of ${\sf ACA}_0$ such that
the arithmetical consequences of $A$ are precisely the consequences of $U$. Similarly for the pair {\sf ZF} and {\sf GB}.

This result was previously proven by Robert van Wesep in his paper \cite{wese:sati13}. 
\end{example}

\begin{example}
By Parsons' result $\mathrm{I}\Sigma_1$ is $\Pi^0_2$-conservative over {\sf PRA}.\footnote{We consider a version
of {\sf PRA} in the original arithmetical language here.} Since, over {\sf EA}, we have $\Sigma_m$-truth predicates.
It follows that, for every $m$, we have an finite extension $A_m$ of {\sf PRA} that is $\Sigma_m$-conservative.  
We can easily arrange that these extensions become strictly weaker when $m$ grows.
\end{example}

\noindent
We refer the reader to \cite{viss:ques17} for a number of results in the same niche using a different methodogy.

\section{Standardness Regained}\label{stare}
Finiteness is Predicate Logic's nemesis. However hard Predicate Logic
tries, there is no way it can pin down the set of standard numbers. What happens
when we invert the question? \emph{Are there theories that interpret some basic arithmetic
that do not have models in which the standard numbers are interpretable?}
The answer is a resounding \emph{yes}. For example, ${\sf PA} + {\sf incon}({\sf PA})$
 has no models in which the standard numbers are interpretable.
More generally, consider \emph{any} recursively enumerable consistent 
theory $U$ with signature $\Theta$. Suppose the signature of arithmetic is $\Xi$. Then,
the theory $U +\verz{(\bigwedge ({\sf S}^1_2)^\tau \to {\sf incon}^\tau(U)) \mid  \tau:\Xi \to \Theta}$ is consistent
and does not have any models that have an internal model isomorphic to the standard numbers.\footnote{We assume that the axioms
of identity are part of the axiomatization of ${\sf S}^1_2$.}

The situation changes when we put some restriction on the complexity of the axioms of the theory.
The classical work concerning this idea  the beautiful paper  by Kenneth McAloon \cite{mcal:comp78}. McAloon 
shows that arithmetical theories with axioms of restricted complexity \emph{that are consistent with
{\sf PA}} always have a model in which the standard integers are definable.
McAloon's work was further extended by Zofia Adamowicz, Andr\'es Cord\'on-Franco and Felix Lara-Mart\'{\i}n. See \cite{adam:exis16}.

Our aim in this paper is to find an analogue of McAloon's Theorem that works for
all sequential theories. We prove a result that is more general in scope but, at the same time, substantially weaker in its statement.
We show that any consistent restricted sequential theory $U$ has a model
in which the \emph{intersection of all definable cuts} is isomorphic to the
standard natural numbers. This intersection is not generally itself definable in the model.
We will show that \emph{the intersection of all definable cuts} is a good notion that, for sequential theories,
is not dependent on the original choice of the interpretation of number theory. 
 
\subsection{The Intersection of all Definable Cuts}
In this subsection, we establish that \emph{the intersection of all definable cuts} is a good notion.

Consider a sequential model $\mathcal M$. 
Let $\mathcal N$ be a $\mathcal M$-internal model satisfying ${\sf S}^1_2$. 
Let  $\mathcal J_{\mathcal M,\mathcal N}$ be the intersection of all
$\mathcal M$-definable $\mathcal N$-cuts in $\mathcal M$. 

Now consider two  $\mathcal M$-internal models $\mathcal N$ and $\mathcal N'$ satisfying ${\sf S}^1_2$. 
By a result of Pavel Pudl\'ak (\cite{pudl:cuts85}), there is an $\mathcal M$-definable isomorphism $\mathcal F$ between an 
$\mathcal M$-definable cut $\mathcal I$ of $\mathcal N$ and an $\mathcal M$-definable cut $\mathcal I'$ of 
$\mathcal N'$.
It is easily seen that $\mathcal F$ restricted to $\mathcal J_{\mathcal M, \mathcal  N}$ is an isomorphism between 
$\mathcal J_{\mathcal M, \mathcal  N}$ and $\mathcal J_{\mathcal M, \mathcal  N'}$.

Suppose $\mathcal G$ and $\mathcal H$ are two $\mathcal M$-definable partial functions between $\mathcal N$ and
$\mathcal N'$
 such that the restrictions of
$\mathcal G$ and $\mathcal H$ to $\mathcal J_{\mathcal M, \mathcal  N}$ 
commute with zero and successor. Then it is easy to see that $\mathcal G$ and $\mathcal H$
are extensionally equal isomorphisms between
$\mathcal J_{\mathcal M, \mathcal  N}$ and $\mathcal J_{\mathcal M, \mathcal  N'}$.
Thus, in a sense, there is a unique definable isomorphism 
 between 
$\mathcal J_{\mathcal M, \mathcal  N}$ and $\mathcal J_{\mathcal M, \mathcal  N'}$.

The above observations justify the notation $\mathcal J_{\mathcal M}$ for
$\mathcal J_{\mathcal M, \mathcal  N}$ modulo isomorphism.

We note that, in the definition of $\mathcal J_{\mathcal M}$ it does not matter 
 whether we allow parameters in the definition of  the cuts.
Every  cut  with parameters  has a parameter-free shortening. Suppose $\mathcal I$ is an $\mathcal N$-cut
that is given by $I(x,\vec b\,)$.
Then,
\[  I^\ast (x) := \forall \vec y\; ({\sf cut}(\verz{ z \mid I(z,\vec y\,)}) \to  I(x,\vec y\,))\]  defines a cut $\mathcal I^\ast$ that is a shortening of 
$\mathcal I$.

\begin{remark}
What happens if the sequence scheme $\mathcal S$ itself involves parameters? In \cite{viss:what13}, it is shown that
these parameters can eliminated  by raising the dimension of the interpretation. Since the standard development of an interpretation of
${\sf S}^1_2$ in a sequential theory does not involve parameters, it follows that even in a sequential theory with a sequence scheme
involving parameters, there is an interpretation of ${\sf S}^1_2$ that is parameter-free! However, the cost of this fact is that
there may be no such interpretation that is one-dimensional.
\end{remark}

\noindent
Before going on with the main line of our story, I want to give some basic facts about $\mathcal J_{\mathcal M}$
in order to place it in perspective.

\begin{theorem}
Suppose $\mathcal M$ is a sequential model.
We have:
 \[\mathcal J_{\mathcal M} \models {\sf EA} + \mathrm{B}\Sigma_1 + \verz{{\sf con}_n(A) \mid \mathcal M \models A}.\]
\end{theorem}

\begin{proof}
Suppose $N$ is, as before,  the interpretation given by the sequence scheme $\mathcal S$ that defines an internal ${\sf S}^1_2$-model of $\mathcal M$. 
We will consider $\mathcal J_{\mathcal M}$ as the intersection of all $N$-cuts. We again write $\Im_n$ for $\Im_n^{\mathcal S}(\Theta)$.

Since there is an $N$-cut $J$ such that on $J$ we have
$W := \mathrm{I}\Delta_0 + \Omega_1 + \mathrm{B}\Sigma_1$ and since this property is downwards
preserved, we find that  $\mathcal J_{\mathcal M} \models W$. 

Suppose $\mathcal M \models A$. 
Without loss of generality, we can assume that  $n \geq \rho_0(A)$. Clearly, we have $\mathcal M \models {\sf con}^{\Im_n}_n(A)$.
Hence, by downwards persistence,  $\mathcal J_{\mathcal M} \models {\sf con}_n(A)$.

Finally, consider any $a$ in  $\mathcal J_{\mathcal M}$. 
Consider any $N$-cut $I$. There is an $N$-cut $I'$ such that for every $b$ in $I'$, we have $2^b$ is in $I$.
Since $a$ is in $I'$, it follows that $2^a$ is in $I$. Since $I$ was arbitrary, we have $2^a$ is in $\mathcal J_{\mathcal M}$.
\end{proof}

\noindent 
Let $\mathfrak J_U := {\sf Th}(\verz{\mathcal J_{\mathcal M} \mid \mathcal M \models U})$. Then, we have:

\begin{corollary}
$\mathfrak J_U \vdash  {\sf EA} + \mathrm{B}\Sigma_1 + \mho_U$.
\end{corollary}

\noindent
The next theorem is a kind of overspill principle.

\begin{theorem}
Suppose $\mathcal M$ is a sequential model and $M$ defines an internal ${\sf S}^1_2$-model of $\mathcal M$. 
We treat $\mathcal J_{\mathcal M}$ as the intersection of all $M$-cuts.
Let  $B(x)$ be any formula. We have:
\begin{center}
\textup(for all $b$ in $\mathcal J_{\mathcal M}$,   $\mathcal M \models B(b)$\textup) iff, for some $M$-cut $J$,
$\mathcal M \models \forall x\in J\, B(x)$.
\end{center}
\end{theorem}

\begin{proof}
The right-to-left direction is trivial. 
Suppose for all $b$ in $\mathcal J_{\mathcal M}$, we have $\mathcal M \models B(b)$. 
Let $X :=
\verz{x \in \delta_M \mid \forall y <^M x\, B(y)}$. If $X$ is closed under successor, then we
can shorten $X$ to an $M$-cut $J$ for which we have $\forall x\in J\, B(x)$, and we are done.
Otherwise, there is a $c$ such that $\neg\, B(c) \wedge \forall y<c\, B(y)$.
Since $c$ cannot be in $\mathcal J_{\mathcal M}$, it follows that there
is a cut $J$ below $c$.
\end{proof}

\noindent
Our overspill principle immediately gives:

\begin{theorem}
Suppose $\mathcal M$ is a sequential model and $M$ defines an internal ${\sf S}^1_2$-model of $\mathcal M$. 
Let  $P$ be a $\Pi_1$-formula. We have:
\begin{center}
$\mathcal J_{\mathcal M} \models P$ iff, for some $M$-cut $J$,
$\mathcal M \models P^J$.
\end{center}
\end{theorem}

\begin{corollary}
Let  $P$ be a $\Pi_1$-formula. Then, $\mathfrak J_U \vdash P$ iff, for some $M:{\sf S}^1_2 \lhd U$, we have
$U \vdash P^M$.
\end{corollary}

\begin{proof}
The proof of the corollary is by a simple compactness argument. 
\end{proof}

\begin{question}
Any further information on $\mathfrak J_U$ would be interesting. For example, what are the possible complexities of
$\mathfrak J_U$ for recursively enumerable sequential theories $U$?
\end{question}

\subsection{$\omega$-models}
Before proceeding, we briefly reflect on the notion of {\omodel}. 
The common practice is to say, e.g.,  that $\mathcal M$ is an {\omodel} of {\sf ZF} if the von Neumann numbers
of $\mathcal M$ are (order-)isomorphic to $\omega$. Of course, there are other interpretations $M$ of arithmetic in
{\sf ZF}. However, we have the feature that if the $M$-numbers are isomorphic to $\omega$, then so are
the von Neumann numbers ---but not \emph{vice versa}. If we consider {\sf GB} in stead of {\sf ZF} we do not know
whether this feature is preserved. It is conceivable that, in some model, a definable cut of the von Neumann
numbers is isomorphic to $\omega$ and the von Neumann numbers are not. 

It seems to me that the proper codification of the common practice would be to say that 
an {\omodel} is not strictly a model but a pair $\tupel{\mathcal M,M}$ of a model and
an interpretation $M$ of a suitable arithmetic in $\mathcal M$ such that $\widetilde M(\mathcal M)$ is
isomorphic to $\omega$.   

 Of course there is the option of existentially quantifying out the choice of the interpretation of arithmetic.
Let's say that $\mathcal M$ is an {\pmodel}, if for some $M$, $\tupel{\mathcal M,M}$ is an
{\omodel}.

Finally, in the sequential case, there is a third option.
We define: a sequential model $\mathcal M$ is \emph{an {\qmodel}} if $\mathcal J_{\mathcal M}$ is isomorphic to the standard numbers.
({\sf i} stands for: intersection.) 
In other words, $\mathcal M$ is an {\qmodel} if, for some interpretation $M$ of ${\sf S}^1_2$,  for
 every non-standard element $a$, there is an $\mathcal M$-definable $M$-cut $I$ such that $I < a$.
 
 \medent
We have the following property of {\qmodel}s.

\begin{theorem}\label{graafsmurf}
Suppose $\mathcal M$ is a sequential {\qmodel} and $M:{\sf S}^1_2\lhd \mathcal M$. 
Let $X$ be a parametrically definable class of $M$-numbers. Suppose $\omega \subseteq X$.
\textup{(}We confuse the standard part of $M$ with $\omega$.\textup{)}
Then there is a $\mathcal M$-definable $M$-cut $J$ such that $J \subseteq X$.
\end{theorem}

\begin{proof}
Suppose $\omega \subseteq X$.
Consider the class $Y:= \verz{a\in N \mid \forall b \leq a \, b \in X}$.
In case $Y$ is closed under successor we can shorten it to a definable cut, and we are done.
In case $Y$ is not closed under successor, there is an $a_0$ such that $\forall b \leq a_0\, b \in X$ but
${\sf S}a_0 \not \in X$. By our assumption $\omega < a_0$. Hence there must be a definable cut $I$
with $\omega \leq I < a_0$. So,  $I \subseteq X$.
\end{proof}

\subsection{The Main Result}
If we are content with the countable case, our main result is a simple application of
the Omitting Types Theorem. We first give this easier proof.  

\begin{theorem}\label{azrael}
Let $U$ be a consistent restricted sequential theory. Here $U$ may be of any complexity. We allow countably many constants in $U$.
Then, $U$ has a model $\mathcal M$ in which $\mathcal J_{\mathcal M}$ is isomorphic to the standard natural numbers.
\end{theorem}

\begin{proof}
We fix a sequence scheme $\mathcal S$ for $U$. We work with the interpretation $N$ of ${\sf S}^1_2$ provided by this scheme.
Suppose that in all countable 
$U$-models the type \[ \mathcal T(x) :=  \verz{ x \neq \underline n \mid  n\in \omega } \cup \verz{ x \in \Im_n \mid n \in \omega}\] is realized.
Then, by the Omitting Types Theorem, there is a formula $A(x)$, such that, for a fresh constant $c$, we have (i)
$U+A(c)$ is consistent and (ii) $U + A(c) \vdash c\neq \underline n$, for each $n\in \omega$,  and (iii)
$U+A(c) \vdash c \in \Im_n$, for each $n\in \omega$. 

We apply Theorem~\ref{ijdelesmurf} to (i) and (ii) obtaining that, for some $n^\ast$, the theory $U+A(c) + \forall x\in \Im_{n^\ast}\, c \neq x$ is
consistent. However, this directly contradicts (i) and (iii).

We may conclude that there is a countable model $\mathcal M$ in which $\mathcal T(x)$ is omitted. Clearly, this tells us that $\mathcal J_{\mathcal M}$ is
isomorphic to the standard numbers.
\end{proof}

\noindent We proceed to prove the stronger version of our theorem where the restriction to countability is lifted.
We first prove a Lemma.

\begin{lemma}\label{fitnesssmurf}
Let $\mathcal M$ be any sequential model of signature $\Theta$ with domain $M$.
Let $\mathcal S$ be a sequence scheme for $\mathcal M$.  As usual, $N$ is the interpretation of
${\sf S}^1_2$ given by the sequence scheme.

Let $k$ be any number. We note that the $\rho_0$-complexity of the axioms of ${\sf AS}^+$ is a fixed number, say $s$. So the sequentiality of
$\mathcal M$ is witnessed by the satisfaction of a sentence $D$ of complexity below $s+ \rho_0(\mathcal S)$. 
Let $n := {\sf max}(k,s+ \rho_0(\mathcal S),  1 + \constantref{celeven}, \rho_0(\mathcal S) + \constantref{celeven})$.

Then, $\mathcal M$ has a sequential $\Delta^\ast_k$-elementary extension $\mathcal K$ with sequence scheme
$\mathcal S$ such that  $\Im_{n}^{\mathcal K} \cap M = \omega$.
\end{lemma}

\begin{proof}
Without loss of generality we may assume that $k \geq s+ \rho_0(\mathcal S)$, so that
${\sf Th}_{\Delta^\ast_k(M)}$ is sequential.
Let $\Gamma := {\sf Th}_{\Delta^\ast_k(M)}(\mathcal M)$. 
We claim that  \[ \Gamma^\ast := \Gamma + \verz{ \mathcal \Im_{n} < m \mid \mathcal M \models m \in \delta_N \wedge \omega <^N m }\] 
is consistent (for $n$ as given in
the statement of the Lemma).
If not, then for some nonstandard  $m_0, \ldots, m_{\ell-1}$ in $M$, we have 
\[ \Gamma \vdash {m_0 \in \Im_n}\; \vee \ldots \vee\; {m_{\ell-1} \in \Im_n}.\]
 Let $m$ be the minimum of the $m_i$.
We find $\Gamma \vdash m \in \Im_n$. On the other hand, the theory $\Gamma + 0<m,\; 1<m, \ldots$ is consistent.
Hence, by Theorem~\ref{ijdelesmurf}, we have that $\Gamma +  \Im_n < m$ is consistent. A contradiction.

Let $\mathcal K$ be a model of $\Gamma^\ast$. Clearly,  in $\mathcal K$, we have that $\Im_n$ is below all
non-standard elements inherited from $\mathcal M$ (but, of course, not necessarily below new non-standard elements). Also
$\mathcal K$ is, by construction, a $\Delta^\ast_n$-elementary extension. 
Finally, since we have chosen  $n \geq s+ \rho_0(\mathcal S)$, the model
$\mathcal K$ is again sequential with the same sequence scheme.
\end{proof}

\noindent
With the Lemma in hand, we can now prove the promised theorem using a limit construction.

\begin{theorem}\label{gargamel}
Let $\mathcal M$ be any sequential model. Then, for any $k$,
$\mathcal M$ has a $\Delta_k^\ast$-elementary extension $\mathcal K$ in which $\mathcal J_{\mathcal K}$ is \textup(isomorphic to\textup) $\omega$.
\end{theorem}

\begin{proof}
Let $\mathcal S$ be a sequence scheme for $\mathcal M$. We work with the numbers $N$ provided
by this scheme.
Let $s$ be as in Lemma~\ref{fitnesssmurf}. We take:
\begin{itemize}
\item
 $n_0 := {\sf max}(k,s+ \rho_0(\mathcal S))$.
 \item
 $n_{j+1} := {\sf max}(\rho_0(\Im_j)+1,s+ \rho_0(\mathcal S),  1 + \constantref{celeven}, \rho_0(\mathcal S) + \constantref{celeven})$.
 \end{itemize}
 
 \noindent 
(We note that $\rho_0(\Im_j) \approx \constantref{czero}j + \constantref{cten}$ and that
 for $j>0$, we have $n_{j+1} := \rho_0(\Im_j)+1$.)

\medent
We construct a chain of models $\mathcal M_i$.
Let $\mathcal M_0 = \mathcal M$. Suppose we have constructed
$\mathcal M_j$. We now take as $\mathcal M_{j+1}$ a model that
is a $\Delta^\ast_{n_{j+1}}$-elementary extension of $\mathcal M_j$ such that
$\Im_{n_j}^{\mathcal M_{j+1}} \cap M_j = \omega$. 

Let $\mathcal K$ be the limit of $(\mathcal M_i)_{i\in \omega}$. Consider any non-standard element $a$ in $\widetilde N(\mathcal K)$.
We have to show that there is a $\mathcal K$-definable cut below it. Suppose $a$ occurs in $\mathcal M_j$.
 We have, by the construction of our sequence, that (\dag)  $\Im_{n_j}^{\mathcal M_{j+1}} < a$. 
By the fact that all $\mathcal M_s$, with $s > j+1$ are $\Delta^\ast_{n_{j+1}}$-elementary extensions of $\mathcal M_{j+1}$, it follows that
(\dag) is preserved to the limit:   $\Im^{\mathcal K}_{n_{j}+1} < a$. 
\end{proof}

\noindent 
From Theorem~\ref{gargamel}, we have immediately the desired strengthening of Theorem~\ref{azrael}.

\begin{theorem}\label{pierewiet}
Let $U$ be a consistent restricted sequential theory. Here $U$ may be of any complexity. We allow a number of  constants in $U$ of any cardinality.
Then $U$ has a model $\mathcal M$ in which $\mathcal J_{\mathcal M}$ is isomorphic to the standard natural numbers.
\end{theorem}

\begin{remark}
From Theorem~\ref{pierewiet}, we retrace our steps and derive a less explicit form of Theorem~\ref{ijdelesmurf}.
Let $U$ be a restricted sequential theory and consider any $C(x)$. Suppose $V := U + \verz{C(\underline n) \mid n \in \omega}$ is
consistent. Let $\mathcal M$ be a model of $V$ in which $\mathcal J_{\mathcal M}$ is isomorphic to the standard natural numbers.
By Theorem~\ref{graafsmurf}, there is a definable $\mathcal M$-cut $J$ so that, in $\mathcal M$, we have $\forall x \in J \,C(x)$.

Now $J$ is a definable cut in $\mathcal M$, but it need not automatically be a cut in $U$. There is a standard
trick to remedy that.
We define $J^\circ := J \tupel{{\sf cut}(J)} N$. Clearly, $J^\circ$ is a definable
cut in $U$. Moreover, in the context of $\mathcal M$, the cuts $J$ and $J^\circ$ coincide.
 By the above considerations, it follows that $U+ \forall x\in J^\circ\, Cx$ is consistent, where $J^\circ$ is a $U$-definable cut.
\end{remark}

\begin{remark}
Consider any model $\mathcal M$. We define ${\sf DEF}(\mathcal M)$ as the class of (parametrically) definable classes of
$\mathcal M$. We define  ${\sf DEF}^{-}(\mathcal M)$ as the class of classes over $\mathcal M$ that are definable without
parameters. 
Also, ${\sf DEF}_n(\mathcal M)$ is the class (parametrically) definable $n$-ary relations 
and similarly for the parameter-free case.

It would seem that Theorem~\ref{gargamel} gives us information about possible sequential models of the form
$\tupel{\mathcal M, {\sf DEF}(\mathcal M)}$, since ${\mathcal J}_{\mathcal M}$ is definable in $\tupel{\mathcal M, {\sf DEF}(\mathcal M)}$.
However, this is not so, since we have a much stronger result for the models $\tupel{\mathcal M, {\sf DEF}(\mathcal M)}$, where $\mathcal M$ is sequential.

We assume that $\mathcal M$ has finite signature, where we may allow an infinity of constants.  In each model $\tupel{\mathcal M, {\sf DEF}(\mathcal M)}$,
where $\mathcal M$ satisfies these demands,
the natural numbers are definable. The argument is simple. Let ${\sf comm}_x(X)$ mean that $X$ satisfies the commutation conditions
for satisfaction of $\Delta_x^\ast$-formulas in $\mathcal J_{\mathcal M}$.
Consider, in $\tupel{\mathcal M, {\sf DEF}(\mathcal M)}$
the class $Y := \verz{ x \in N\mid \exists X\, {\sf comm}_x(X)}$. Clearly, each standard $x$ is in $Y$. If a non-standard number $b$ would be
in $Y$, the defining formula for the witnessing $X$ would violate Tarski's Theorem of the undefinability of truth for $\mathcal M$.\footnote{Ali Enayat
tells me that the basic idea of this argument is originally due to Mostowski.}

If the sequence scheme for $\mathcal M$ is parameter-free, then the same argument works for $\tupel{\mathcal M, {\sf DEF}^-(\mathcal M)}$.
If the sequence scheme contains parameters, we can make the argument work for $\tupel{\mathcal M, {\sf DEF}^{-}_n(\mathcal M)}$, for sufficiently
large $n$.

Finally, note that if $\mathcal M$ is a non-standard model of Peano Arithmetic, then $\mathcal J_{\mathcal M}$ is simply isomorphic to $\mathcal M$ itself.
Thus, adding $\mathcal J_{\mathcal M}$ (viewed as intersection of all cuts on the identical interpretation of ${\sf S}^1_2$) to $\mathcal M$ does not
increase the expressiveness of the language. This consideration shows that adding the definable sets can be more expressive than adding
$\mathcal J_{\mathcal M}$ (as intersection of the cuts for a given interpretation of ${\sf S}^1_2$).
\end{remark}

\section{Reflection}\label{refl}
If we apply Theorem~\ref{grotesmurfres} to a formula of a special form, we get a  reflection principle.

\begin{theorem}\label{sterkesmurf}
Consider any consistent, restricted, sequential theory $U$ with sequence scheme $\mathcal S$. Let $N$ be the interpretation of
the numbers provided by $\mathcal S$. Let $m_0$ be the bound for $U$ and let $m_1$ be any number.
Let $n := {\sf max}(m_0, m_1 + \rho_0(\mathcal S) + \constantref{celeven})$.  
 
Then, for every $\Sigma_2$-sentence $C$ of the form $C = \exists x\, C_0(x)$, where $C_0$ is $\Pi_1$ and $\rho_0(C) \leq m_1$,
we have: if $U \vdash \exists x \in \Im_n \, C^N_0(x)$, then $C$ is true.
\end{theorem}

\begin{proof}
Under the assumptions of the theorem, we suppose $U \vdash \exists x \in \Im_n \, C^N_0(x)$.
By Theorem~\ref{grotesmurfres},
there is  a $k$ such that $U \vdash \bigvee_{q\leq k}  C^N_0(\underline q)$.
Suppose $C$ is false. Then, for each $q\leq k$, we have $\neg\, C_0(q)$. Hence,
by $\Sigma_1$-completeness, for each $q\leq k$, we have $U \vdash \neg\, C_0^N(\underline q)$. It follows that $U \vdash \bot$.
\emph{Quod non}.
\end{proof}

\noindent 
We note that the above proof uses $\Sigma^0_1$-collection in the metalanguage.

\medent
If we take the formulas still simpler we can improve the above result. 
We fix a logarithmic cut  ${\sf S}^1_2$-cut $\widetilde J$.
There is a $\Sigma_1$-truth predicate, say {\sf True}, for $\Sigma_1$-sentences, such that,
for any $\Sigma_1$-sentence $S$, we have
${\sf S}^1_2 \vdash {\sf True}(S) \to S$ and ${\sf S}^1_2 \vdash S^{\widetilde J} \to {\sf True}(S)$.
(See e.g. \cite{haje:meta91}, Part V, Chapter 5b for details.)

\begin{theorem}\label{supersmurf}
Consider any consistent, restricted, sequential theory $U$ with sequence scheme $\mathcal S$ and bound $m$.
Let $n:= {\sf max}(m, \rho_0({\sf True}(x)) +\rho_0(\mathcal S)+\constantref{celeven}+1)$.
For all  $\Sigma_1$-sentences $S$, we have: if
$U \vdash S^{\widetilde J \Im_n}$, then  $S$ is true.
\end{theorem}

\begin{proof}
We have:
\begin{eqnarray*}
U \vdash S^{\widetilde J\Im_n} & \To & U \vdash {\sf True}^{\Im_n}(S) \\
& \To & {\sf True}(S) \text{  is true} \\
 & \To & S \text{ is true}.
 \end{eqnarray*}
 The second step is by theorem~\ref{sterkesmurf}. 
\end{proof}

\section{Degrees of Interpretability}\label{degrees}

In this section, we apply our results to study the joint degree structure of local and global interpretability for recursively
enumerable sequential theories. The main result of this section is a characterization of finite axiomatizability in terms
of the double degree structure. 
We will study the degree structures as partial pre-orderings.

\subsection{The Basic Idea}
The degree structures we are interested in are \emph{the degrees of global interpretability of recursively enumerable
sequential theories} ${\sf Glob}_{\sf seq}$ and  \emph{the degrees of local interpretability of recursively enumerable
sequential theories} ${\sf Loc}_{\sf seq}$.  It is well known that both structures are distributive
lattices.   We have the obvious projection functor
$\pi$ from ${\sf Glob}_{\sf seq}$ onto ${\sf  Loc}_{\sf seq}$. 

Let ${\sf Fin}$ be the property of global degrees of \emph{containing a finitely axiomatized theory}.
What we want to show is that, if we start with the pair ${\sf Glob}_{\sf seq}$ and
${\sf  Loc}_{\sf seq}$ and with the projection $\pi$, then we can define {\sf Fin} (using a first-order formula).

The basic idea is simple. Zoom in on a local degree of $U$. This degree contains a
distributive lattice, say, $\mathcal L$ of global degrees. The lattice $\mathcal L$ has a maximum,
to wit $\mho_U$. Does it have a minimum? Well, if there is a finitely axiomatized theory $U_0$
in $\mathcal L$, then its global degree will automatically be the minimum. We will see that
(i)  not in all cases a minimum degree of $\mathcal L$ exists and (ii) if such a minimum exists
it contains a finitely axiomatizable theory. In other words, the mapping $\phi$ with:

\[  \phi(U) \lhd_{\sf glob} V \;\; \Iff \;\; U \lhd_{\sf loc} \pi(V)\]

\noindent
is partial. However, if it has a value, this value contains a finitely axiomatized theory.
So, we can define: ${\sf fin}(U)$ iff, for all $V$, we have $U \lhd_{\sf glob}V$ iff $\pi(U) \lhd_{\sf loc} \pi(V)$.

\begin{question}
Can we define {\sf fin} in ${\sf Glob}_{\sf seq}$ alone?
\end{question}

\begin{remark}
 In the context of local degrees of arbitrary theories with arbitrarily large signatures,
 Mycielski, Pudl\'ak and Stern charactarize loc-finite as the same as \emph{compact} in terms
 of the $\lhd_{\sf loc}$-ordering. This will not work in our context of global interpretability and recursively enumerable
 sequential theories. We briefly give the argument that, in our context, every non-minimal globally finite degree is non-compact.  
 
 Let $A$ be any $\lhd$-non-minimal, finitely axiomatized,
 sequential theory. Let $B_i$ be a enumeration of all finitely axiomatized sequential theories.
 Let $C_0 := {\sf S}^1_2$, We note that, by our assumption, $C_0 \lhdneq A$. Let  $C_{n+1} := B_n$  if  $C_n \lhdneq B_n \lhdneq A$ and
 $C_{n+1} := C_n$ otherwise. Clearly, $A \rhd C_i$, for all $i$. Consider any sequential recursively enumerable theory
 $U$ such that $U \rhd C_i$, for all $i$. Without loss of generality, we may assume that the signatures of $U$ and $A$ are
 disjoint. We easily see that the theory $U^\ast$ axiomatized by $\verz{ (D\vee A) \mid D \text{ is an axiom of } U}$ is the
 infimum in the degrees of global interpretability of $U$ and $A$. So $C_i \lhd U^\ast \lhd A$. Suppose
 $U^\ast \lhdneq A$. In this case, by \cite[Theorem 5.3]{viss:inte17}, we can find a finitely axiomatized sequential $B$ such that $U^\ast \lhdneq B \lhdneq A$.
 let $B = B_j$. Since $C_j \lhd U^\ast \lhdneq B \lhdneq A$, we will have $C_{j+1} := B_j$. A contradiction. It follows that
 $U^\ast \equiv A$. We may conclude that $A \lhd U$. So, $A$ is the supremum of the $C_i$. Clearly, $A$ cannot be the supremum of
 a finite number of the $C_i$. Thus, $A$ is not compact.
\end{remark}

\subsection{Preliminaries}
We consider the recursively enumerable theory $U$. By Craig's Theorem, we can give
$U$ a $\Sigma_1^{\sf b}$-definable axiomatization $X$. Let this axiomatization is given by a $\Sigma_1^{\sf b}$-formula $\eta$.
We define $\thres{U}{n}$ as the theory axiomatized by the axioms of $U$, as given by $\eta$ that are $\leq n$.

An important functor is  $\mho$. We define:  
$\mho_U := {\sf S}^1_2 + \verz{{\sf con}_n(\thres{U}{n}) \mid n \in \omega}$.
We note that $\mho_U$ is extensionally independent of the choice of $\eta$.

One can show that $\mho$ is the right adjoint of $\pi$: 
\[  U \lhd_{\sf glob} \mho_V \;\; \Iff \;\; \pi(U) \lhd_{\sf loc} V.\]
See e.g. \cite{viss:seco11} or \cite{viss:inte17}.

A theory is {\em {\sf glob}-finite}
iff it is  mutually globally interpretable with a finitely axiomatized
theory.  A theory is {\em {\sf loc}-finite}
iff it is  mutually locally interpretable with a finitely axiomatized
theory.
If we enrich ${\sf Glob}_{\sf seq}$ with a predicate {\sf Fin} for
the globally finite degrees, we have a first-order definition of ${\sf  Loc}_{\sf seq}$ over this structure
as follows: \[ U \lhd_{\sf loc} V\;\;  \Iff  \;\; \forall A \in {\sf Fin}\, (A \lhd_{\sf glob} U \To A \lhd_{\sf glob} V).\]

\noindent
 Here is a first basic insight.

\begin{theorem}
The theory $U$ is $a$-finite, for $a\in\verz{\sf glob, loc}$,
 iff $U\equiv_a \thres{U}{n}$, for some $n$. 
\end{theorem}

\begin{proof}
Suppose $U\equiv_a V$, where $V$ is
finitely axiomatized. Clearly, $\thres{U}{n}\rhd V$, for some $n$.
We have: $ {U\supseteq  \thres{U}{n}\,} \rhd V \rhd_a U$, and we are done.
\end{proof}

\subsection{Some Examples}
\noindent
Before formulating and proving our main result, we briefly pause to provide a
few examples.

\begin{example}
The theories $\mathrm{I}\Delta_0$ and ${\sf S}_2=\mathrm{I}\Delta_0+\Omega_1$ are examples of theories
of which the finite axiomatizability is an open problem, but which are,
by an argument of Alex Wilkie, {\sf glob}-finite. They are, for example, mutually interpretable with the finitely
axiomatized sequential theory {\sf AS} and with the finitely axiomatized sequential theory ${\sf PA}^-$.
 \end{example}
 \begin{figure}
 \[
 \begin{tabular}{|c||c|c|c|}\hline
& {\sf loc}-finite & {\sf glob}-finite & fin. axiom. \\ \hline\hline
 {\sf GB}  & $+$ & $+$ & $+$ \\ \hline
$\mathrm{I}\Delta_0$, ${\sf S}_2$ & $+$ & $+$ & $?$ \\ \hline
${\sf GB}^\circ$  & $+$ & $+$ & $-$ \\ \hline
$\mho_{\sf GB}$  & $+$ & $-$ & $-$ \\ \hline
{\sf PA} &  $-$ & $-$ & $-$ \\ \hline
 \end{tabular}
 \]
 \caption{Separating Sequential Local and Global Finiteness}
\end{figure}
 
 \noindent
We  show that any global, recursively enumerable, sequential degree contains an element that is not finitely
axiomatizable.

 \begin{theorem}
 Consider any consistent, sequential, recursively enumerable theory $U$. 
 Then there is a $U^\circ\equiv U$ such that
 $U^\circ$ is sequential and recursively enumerable and not finitely axiomatizable.
 \end{theorem}
 
\begin{proof}
Consider a consistent, sequential and recursively enumerable theory $U$.
In case $U$ is not finitely axiomatizable, we are done, taking $U^\circ:=U$.
Suppose $U$ is finitely axiomatizable, say by a single sentence $A$.
Par abus de langage, we write $A$ also for the theory axiomatized by $x = \gnum{A}$.

By Theorem~\ref{supersmurf}, we can find $M:{\sf S}^1_2 \lhd A$
for which $A$ is $\Sigma^0_1$-sound. 
Consider:
  \[U^\circ:= A+\verz{({\sf con}(A)\to {\sf con}^{n+1}(A))^M \mid n\in \omega}.\]
 We have, by Feferman's version of the Second Incompleteness Theorem:
\[ A \subseteq U^\circ \subseteq (A+{\sf incon}^M(A)) \lhd A. \]
So $U^\circ\equiv  A$.
Suppose $U^\circ$ were finitely axiomatizable.
Then, we would  have, for some $n>0$, 
\[A+ ({\sf con}(A)\to {\sf con}^n(A))^M \vdash
({\sf con}(A)\to {\sf con}^{n+1}(A))^M.\]
Hence $A\vdash (\opr^{n+1}_A\bot \to \opr^n_A\bot)^M$. So, by 
L\"ob's Theorem, $A\vdash (\opr^n_A\bot)^M$, contradicting the $\Sigma^0_1$-soundness of
$A$ w.r.t. $M$. 
\end{proof}

\noindent Here is a sufficient condition for failure to be {\sf loc}-finite.
We define:
\begin{itemize}
\item
 $\mho^+_U := {\sf S}^1_2 + \verz{{\sf con}(U\restriction n) \mid n \in \omega}$.\\
  Here $U \restriction n$ is defined with respect to a chosen $\Sigma^{\sf b}_1$-formula $\eta$ that represents the axiom set of $U$.
 \end{itemize}
We call $U$  \emph{strongly {\sf loc}-reflexive} if $U\rhd \mho^+_U$.
We note that e.g. {\sf PRA} is an example of a strongly {\sf loc}-reflexive theory.

\begin{theorem}
Suppose that $U$ is strongly {\sf loc}-reflexive. Then, $U$ is not
{\sf loc}-finite.
\end{theorem}

\begin{proof}
Suppose $U\rhd_{\sf loc}\mho^+_U$. Suppose, to obtain a contradiction, that
$U$\/ is {\sf loc}-finite. Then,  for some $i$, we have $\thres{U}{i}\rhd_{\sf loc} U$.
Moreover, $U\rhd ({\sf S}^1_2+{\sf con}(\thres{U}{i}))$. So,
$\thres{U}{i}\rhd ({\sf S}^1_2+{\sf con}(\thres{U}{i}))$, contradicting the second 
incompleteness theorem.
\end{proof}

\subsection{Characterizations}
 The following characterization of {\sf loc}-finiteness may look
like an \emph{obscurum per obscurius}. However, it is a central tool in what follows.

\begin{theorem}\label{charfindue}
Suppose $U$ is sequential. The variable $I$ will range over ${\sf S}^1_2$-definable cuts. The following are equivalent:
\begin{enumerate}[1.]
\item
$U$ is {\sf loc}-finite.
\item
$\exists i\;\forall j{\geq}i\;\exists I\;\; {\sf S}^1_2+{\sf con}_i(\thres{U}{i})\vdash {\sf con}_j^I(\thres{U}{j})$.
\item
$\exists i\;\forall j{\geq}i\;\exists I\;\; {\sf S}^1_2+{\sf con}_j(\thres{U}{j})\vdash {\sf con}_{j+1}^I(\thres{U}{(j+1)})$.
\end{enumerate}
\end{theorem}

\begin{proof}
$(1)\To (2)$. Suppose $U$ is {\sf loc}-finite. Then,
 $\thres{U}{i}\rhd_{\sf loc}  U$,  for some $i$. Consider any $j\geq i$. For some $M$, 
we have  $M:{\thres{U}{i} \rhd \thres{U}{j}}$.
Using $M$, we can transform a $\thres{U}{j},j$-inconsistency 
proof into an $\thres{U}{i},k$-inconsistency
proof, for a sufficiently large $k$. Thus, we have 
${\sf S}^1_2\vdash {\sf con}_k(\thres{U}{i}) \to {\sf con}_j(\thres{U}{j})$. On the other hand, for some
$I$, ${\sf S}^1_2+{\sf con}_i(\thres{U}{i})\vdash {\sf con}_k^I(\thres{U}{i})$. This last step can be seen, e.g. from the
fact that a cut-elimination that transforms a $k$-proof to an $i$-proof is multi-exponential (see \cite{buss:situ15}).\footnote{\label{obersmurf}Alternatively, 
we can prove the step by a combination of the Interpretation Existence Lemma (\cite{viss:inte17}) and the local reflexiveness
of $\thres{U}{i}$ that is a direct consequence of Theorem~\ref{summumsmurf} of the present paper.}
 It suffices to take
as $I$ an appropriate multi-logarithmic cut.
We may conclude that
${\sf S}^1_2+{\sf con}_i(\thres{U}{i})\vdash {\sf con}_j^I(\thres{U}{j})$.

\medent
$(2)\To (3)$ and $(3)\To (2)$ are trivial. 

\medent
For $(2)\To (1)$, one shows that, for $i$ as promised,
${\sf S}^1_2+{\sf con}_i(\thres{U}{i})$ is mutually  locally interpretable with $U$. 
\end{proof}

\noindent
We note that in the above proof, the precise point where where sequentiality is used, is the insight that $U$ interprets ${\sf S}^1_2+{\sf con}_i(\thres{U}{i})$.
Thus, only $(2)\To (1)$ depends on sequentiality.\footnote{This is not true anymore when we employ the argument of Footnote~\ref{obersmurf}.}

\medent
By a result of Wilkie and Paris (\cite{wilk:sche87}, see also \cite{viss:seco11}), we have, for $\Sigma_1$-sentences $P$ and $Q$, that
${\sf EA} + P \vdash Q$ iff, for some ${\sf S}^1_2$-cut $I$, we have ${\sf S}^1_2+P \vdash Q^I$. Hence, it follows that: 

\begin{corollary}
Suppose $U$ is sequential.  The following are equivalent:
\begin{enumerate}[1.]
\item
$U$ is {\sf loc}-finite.
\item
$\exists i\;\forall j{\geq}i\;\;  {\sf EA}+{\sf con}_i(U)\vdash {\sf con}_j(U)$.
\item
$\exists i\;\forall j{\geq}i\;\; {\sf EA}+{\sf con}_j(U)\vdash {\sf con}_{j+1}(U)$.
\end{enumerate}
\end{corollary}

\subsection{The Main Theorem}
Consider the partial preorder, say $\mathcal L$,
of sequential degrees of global interpretability contained
in a given degree of local sequential interpretability. Note that $\mathcal L$ is closed under
suprema and infima. Suppose our given local degree is
 not {\sf loc}-finite. Then, $\mathcal L$ does not have a minimum.
This insight is formulated in the following theorem. 

\begin{theorem}\label{nomin}
Suppose that  $U$ is a recursively enumerable, sequential theory that is \emph{not}
{\sf loc}-finite. Then, there is a theory $\widetilde U$, such that
$\widetilde U\equiv_{\sf loc}U$, but  $\widetilde U\nrhd_{\sf glob} U$.
\end{theorem}

\begin{proof} 
Suppose that  $U$ is recursively enumerable, sequential and not {\sf loc}-finite. 
Let the signature of $U$ be
$\Theta$ and let a sequence scheme
for $U$ be $\mathcal S$. As usual $N$ is the standard interpretation of ${\sf S}^1_2$ in $U$ given by $\mathcal S$.
We write $\Im_i$ for $\Im^{\mathcal S}_i(\Theta)$.
The complexities of the $\Im_i$ are estimated by $\constantref{czero} i +\constantref{cten}+\rho_0(\mathcal S)$. 
Taking $\constant{ctwelve} :=  \constantref{cten}+\rho_0(\mathcal S)$, our estimate becomes:
$\constantref{czero} i +\constantref{ctwelve}$. 
By the choice of the $\Im_i$ we have: $U \vdash {\sf con}^{\Im_i}_i(\thres{U}{i})$.
 
\medent
Consider ${\sf S}^1_2$. Say the signature of ${\sf S}^1_2$ is $\Xi$ and a sequence scheme
that interprets the numbers identically is $\mathcal T$. We write $\mathcal I_i$ for $\Im^{\mathcal T}_i(\Xi)$.
The complexities of the $\mathcal I_i$ are estimated by 
$\constantref{czero} i +\constantref{cten}+\rho_0(\mathcal T)$.
Taking $\constant{cthirteen} :=  \constantref{cten}+\rho_0(\mathcal T)$, our estimate becomes:
$\constantref{czero} i +\constantref{cthirteen}$. 

\medent
We define  the theory $\widetilde U$ as follows:
Let  Let $F(i):=  (\constantref{czero} +1) i +1$.\footnote{It is a sport to take the choice of $F$ as sharp
as possible. The argument below becomes a bit more relax if we take
$F(i) := i^2 +1$ and just keep track of linear dependencies.}
 \[ \widetilde U := {\sf S}^1_2+\verz{{\sf con}_i^{\mathcal I_{F(i)}}(\thres{U}{i})\mid i\in\omega}.\]
Clearly, $\widetilde U \equiv_{\sf loc} U$. 

\bident 
Suppose, to obtain a contradiction, that, for some $K$, we have
$K:\widetilde U \rhd U$. By Pudl\'ak's theorem (\cite{pudl:cuts85}), 
there is a $\widetilde U$-cut
$J$ of $N$ and a $\widetilde U$-cut $J'$ of $K\circ N$ and a $\widetilde U$-definable isomorphism $G$ between $J$
and $J'$. We define $J_i := G^{-1}[\Im^K_i \cap J']$. We clearly have: $J_i$ is a $\widetilde U$-cut and 
$\widetilde U \vdash {\sf con}^{J_i}_i(U)$. We note that:
\[ x\in J_i :\iff \exists y\, (Gxy \wedge (y \in \Im_i)^K \wedge y \in J'),\]
 so
 that $\rho_0(J_i)$ is estimated by a linear term of the form $\constantref{czero} i + \constant{cfourteen}$.
 (Here $\constantref{cfourteen}$ is dependent on $\rho_0(K)$.)

\bident
Consider any $s$. We will make $s$ more specific in the run of the argument.
We have: $\widetilde U\vdash {\sf con}_{s+1}^{J_{s+1}}(\thres{U}{(s+1)})$. Hence, by compactness, for some $p\geq s$:
\begin{equation}
{\sf S}^1_2+\verz{{\sf con}_i^{\mathcal I_{F(i)}}(\thres{U}{i})\mid i\leq s}+
\verz{{\sf con}_j^{\mathcal I_{F(j)}}(\thres{U}{j})\mid s<j\leq p}
\vdash   {\sf con}_{s+1}^{J_{s+1}}(\thres{U}{(s+1)})
\end{equation} 
Thus, it follows that:
\begin{equation}
{\sf S}^1_2+{\sf con}_s(\thres{U}{s})+{\sf incon}_{s+1}^{J_{s+1}}(\thres{U}{(s+1)})\vdash {\sf incon}_p^{\mathcal I_{F(s+1)}}(\thres{U}{p})
\label{knockdown}
\end{equation}
Since, $U$ is not {\sf loc}-finite, we have, by Theorem~\ref{charfindue},
arbitrarily large $s$'s such that $A_s:= {\sf S}^1_2+{\sf con}_s(\thres{U}{s})+{\sf incon}_{s+1}^{J_{s+1}}(\thres{U}{(s+1)})$
is consistent. We note that $\rho_0(A_s)$ is estimated by $ \constantref{czero} s + \constant{cfifteen}$ for a suitable $\constantref{cfifteen}$.   

Consider any $s$ for which $A_s$ is consistent. 
We remind the reader of Theorem~\ref{sterkesmurf}.
Applied to the case at hand this tells us the following.
Let $n:= {\sf max}(\rho_0(A_s),\constant{csixteen})$ (where $\constantref{csixteen}$ is a fixed constant). Then,
 if
$A_s \vdash {\sf incon}^{{\mathcal I}_n}_p(\thres{U}{p})$, then  ${\sf incon}_p(\thres{U}{p})$ is true. Since we assumed
that $U$ is consistent, it follows that $A_s \nvdash {\sf incon}^{{\mathcal I}_n}_p(\thres{U}{p})$.

We note that $n$ is estimated by $\constantref{czero} s + \constant{cseventeen}$ for a suitable $\constantref{cseventeen}$.
We may choose $s$ large enough so that $F(s+1) > \constantref{czero} s + \constantref{cseventeen}$ and such that $A_s$ is consistent.
It follows that $\mathcal I_{F(s+1)}$ is a subcut of $\mathcal I_n$.
Since, by  Equation~(\ref{knockdown}), $A_s \vdash {\sf incon}_p^{\mathcal I_{F(s+1)}}(\thres{U}{p})$, it follows that  
$A_s \vdash {\sf incon}_p^{\mathcal I_{n}}(\thres{U}{p})$. A contradiction.
\end{proof}

\noindent Note that Theorem~\ref{nomin} implies  that any local sequential degree that is not
{\sf loc}-finite contains an infinity of global degrees. 

\begin{question}
Every element of ${\sf Loc}_{\sf seq}$ contains an extension of ${\sf S}^1_2$ (to wit an element of the form $\mho_U$).
Does every element of ${\sf Glob}_{\sf seq}$ contain an extension of ${\sf S}^1_2$?
\end{question}

\appendix

\section{Parameters}\label{machosmurf}
  In general, interpretations are allowed to have parameters. We will briefly sketch how to add
  parameters to our framework. We first define a translation with parameters.
  The parameters of the translation are given by a fixed sequence of variables $\vec w$ that we
  keep apart from all other variables. A translation is defined as before, but for the fact that now
  the variables $\vec w$ are allowed to occur in the domain-formula and in the translations of the predicate symbols
  in addition to the variables that correspond to the argument places. Officially, we represent a
  translation $\tau_{\vec w}$ with parameters $\vec w$  as a quintuple $\tupel{\Sigma,\delta,\vec w, F,\Theta}$. 
  The parameter sequence may be empty: in this case our interpretation is parameter-free.
    
  An interpretation with parameters  $K:U\to V$ is a quadruple  $\tupel{U,\pi,E,\tau_{\vec w}, V}$, where
  $\tau_{\vec w}:\Sigma_U \to \Sigma_V$ is a translation and $\pi$ is a $V$-formula
  containing at most $\vec w$ free. The formula $\pi$ represents the parameter domain.
   For example, if we interpret the Hyperbolic Plane in the Euclidean Plane via the Poincar\'e interpretation, we
  need two distinct points to define a circular disk. These points are parameters of the construction, the
  parameter domain is $\pi(w_0,w_1) = (w_0\neq w_1)$. (For this specific example, we can also find
  a parameter-free interpretation.) The formula $E$ represents an equivalence relation on the parameter
  domain. In practice this is always pointwise identity for parameter sequences, but for reasons of theory
  one must admit other equivalence relations too.  
  We demand:
  \begin{itemize}
  \item 
  $\vdash \delta_{\tau,\vec w}(\vec v) \to \pi(\vec w)$,
  \item 
  $\vdash P_{\tau,\vec w}(\vec v_0,\ldots,\vec v_{n-1}) \to \pi(\vec w)$.
  \item
  $V\vdash \exists \vec w\, \pi(\vec w)$;
  \item
  $V\vdash E(\vec w,\vec z) \to (\pi(\vec w) \wedge \pi(\vec z))$;
  \item
  $V$ proves that $E$ represents an equivalence relation on the sequences forming the parameter domain;
  \item
  $\vdash E(\vec w,\vec z) \to \forall \vec x\, (\delta_{\tau,\vec w}(\vec x) \iff \delta_{\tau,\vec z}(\vec x))$;
  \item
  $\vdash E(\vec w,\vec z) \to \forall \vec x_0,\ldots,\vec x_{n-1}\;(P_{\tau,\vec w}(\vec x_0,\ldots,\vec x_{n-1}) \iff
  P_{\tau,\vec z}(\vec x_0,\ldots,\vec x_{n-1}) )$;
   \item
  for all $U$-axioms $A$, $V\vdash \forall \vec w\,(\pi(\vec w) \to A^{\tau,{\vec w}})$.
  \end{itemize}
  
  \noindent
  We can lift the various operations in the obvious way.   
    Note that the parameter domain  of
  $N :=M\circ K$ and the corresponding equivalence relation should be:
  \begin{itemize}
  \item
  $\pi_N(\vec w,\vec u_0,\ldots,\vec u_{k-1}) := \pi_M(\vec w) \wedge \bigwedge_{i<k}\delta_{\tau_M}(\vec w,\vec u_i)
  \wedge (\pi_K(\vec u))^{\tau_M,\vec w}$.
  \item
  $E_N(\vec w,\vec u_0,\ldots,\vec u_{k-1},\vec z,\vec v_0,\ldots,\vec v_{k-1}) := \\
  E_M(\vec w,\vec z) \wedge  \bigwedge_{i<k}\delta_{\tau_M}(\vec w,\vec u_i) \wedge
\bigwedge_{i<k}\delta_{\tau_M}(\vec w,\vec v_i) \wedge (E_K(\vec u,\vec v))^{\tau_M,\vec w}$.
  \end{itemize}
  
  \end{document}